\documentclass[leqno]{amsart}

\usepackage{amscd,amssymb}

\usepackage{url}

\usepackage[all,2cell,cmtip]{xy} \UseAllTwocells \SilentMatrices

\usepackage[colorlinks=true]{hyperref}

\newcommand{\I}{\mathbf 1}
\newcommand{\C}{\mathbb C}

\newcommand{\bG}{\mathbb G}
\newcommand{\sC}{\mathcal C}
\newcommand{\sD}{\mathcal D}
\newcommand{\sE}{\mathcal E}
\newcommand{\sF}{\mathcal F}
\newcommand{\sJ}{\mathcal J}
\newcommand{\sM}{\mathcal M}
\newcommand{\sN}{\mathcal N}
\newcommand{\sO}{\mathcal O}
\newcommand{\sR}{\mathcal R}
\newcommand{\sV}{\mathcal V}
\newcommand{\sW}{\mathcal W}
\newcommand{\iso}{\xrightarrow{\,\,\sim\,}}
\newcommand{\nd}{\nobreakdash-\hspace{0pt}}

\renewcommand{\theenumi}{(\arabic{enumi})}

\DeclareMathOperator{\Aut}{Aut}
\DeclareMathOperator{\End}{End}
\DeclareMathOperator{\Ext}{Ext}
\DeclareMathOperator{\Hom}{Hom}
\DeclareMathOperator{\Iso}{Iso}
\DeclareMathOperator{\Mod}{Mod}
\DeclareMathOperator{\MOD}{MOD}
\DeclareMathOperator{\Rep}{Rep}
\DeclareMathOperator{\REP}{REP}
\DeclareMathOperator{\Sp}{Sp}
\DeclareMathOperator{\Spec}{Spec}
\DeclareMathOperator{\tr}{tr}

\newtheorem{thm}{Theorem}[section]
\newtheorem{cor}[thm]{Corollary}
\newtheorem{lem}[thm]{Lemma}
\newtheorem{prop}[thm]{Proposition}

\theoremstyle{definition}

\newtheorem{remark}[thm]{Remark}

\newtheorem*{assump*}{Assumption}

\numberwithin{equation}{section}

\begin{document}

\title{Pullback of principal bundles along proper morphisms}

\author[I. Biswas]{Indranil Biswas}

\address{Department of Mathematics, Shiv Nadar University, NH91, Tehsil Dadri,
Greater Noida, Uttar Pradesh 201314, India}

\email{indranil.biswas@snu.edu.in, indranil29@gmail.com}

\author[P. O'Sullivan]{Peter O'Sullivan}

\address{Mathematical Sciences Institute, The Australian National University,
Canberra ACT 2601, Australia}

\email{peter.osullivan@anu.edu.au}

\thanks{}



\date{}

\dedicatory{}

\begin{abstract}
We study the behaviour of principal bundles under pullback along proper surjective morphisms of either schemes over an algebraically closed
field of characteristic 0 or complex analytic spaces.
\end{abstract}

\maketitle

\tableofcontents

\section{Introduction}

It has been shown \cite[Lemma~4.2]{BisDum22} that a principal bundle with reductive structure group 
over a connected compact complex manifold is almost trivial, 
in the sense that its structure group can be reduced to a finite subgroup, if and only if its pullback along a proper surjective morphism is almost trivial. 
On the other hand if $X$ is a connected scheme proper over an algebraically closed field of characteristic $0$,
it has been shown \cite[Proposition~15.5]{O19} that a principal bundle over $X$ with reductive structure
group is almost minimal, in the sense that its structure group cannot be reduced to any reductive subgroup which is not of finite index,
if and only if its pullback along a surjective finite locally free morphism $f\,:\,X' \,\longrightarrow\, X$ is almost minimal.
In this paper we prove that when $X$ is normal, the almost minimality statement holds for an arbitrary proper surjective morphism $f$.
If further $X'$ is connected and $f$ induces a surjective homomorphism on fundamental groups, then $f$
induces a bijection from isomorphism classes of principal subbundles with reductive structure group of a given principal bundle over $X$ to those of its pullback
along $f$. 
Complex analytic analogues of these results are also proved, which contain in particular the above almost triviality statement.

To describe the results more precisely, we begin with the algebraic case.
Let $k$ be an algebraically closed field of characteristic $0$, and let $X$ be a scheme over $k$.
If $G$ is an affine algebraic group over $k$, recall that a principal $G$\nd bundle over $X$ is a scheme $P$
over $X$ together with a right action of $G$ on $P$ over $X$ such that locally in the \'etale topology, $P$ is
isomorphic over $X$ to $X \times_k G$ with $G$ acting through right translation.
If $G_1$ is a $k$\nd subgroup of $G$, a principal $G_1$\nd subbundle of $P$ is a closed subscheme $P_1$ of $P$
such that the action of $G$ on $P$ restricts to an action of $G_1$ on $P_1$ with $P_1$ a principal $G_1$\nd bundle over $X$.
Two principal $G_1$\nd subbundles of $P$ will be called isomorphic if they are isomorphic as principal $G_1$\nd bundles over $X$.
Any such isomorphism is induced by a unique automorphism of the principal $G$\nd bundle $P$.
In this paper we are mainly concerned with principal subbundles with (not necessarily connected) reductive structure group,
and the behaviour of isomorphism classes of such subbundles under pullback along proper surjective morphisms.

A principal bundle with reductive structure group will be called \emph{minimal} if it has no principal subbundle 
with a strictly smaller reductive structure group.
The importance of such principal bundles in the present context comes from the fact that, under appropriate conditions on $X$, the isomorphism 
classes of principal $G_1$\nd subbundles of a principal $G$\nd bundle $P$ over $X$ for any reductive $k$\nd subgroup
$G_1$ of $G$ are completely determined once a reductive $k$\nd subgroup $G_0$ of $G$ is known for which $P$ has a minimal principal $G_0$\nd subbundle.
Explicitly, if $H^0(X \, , \, \sO_X)$ is a henselian local $k$\nd algebra with residue field $k$,
then the isomorphism classes of principal $G_1$\nd subbundles of $P$ are parametrised by
a finite set \eqref{e:Tset} depending only on $G$, $G_0$ and $G_1$ but not on $X$ or $P$.
This follows from Theorem~\ref{t:push} below, which in the case where $H^0(X \, , \, \sO_X) = k$ is equivalent to 
a result of Bogomolov \cite[p.~401, Theorem~2.1]{Bo94}.

Suppose now that $X$ is locally noetherian and normal, and let $f\,:\,X' \,\longrightarrow\, X$ be a proper surjective morphism.
Then $P$ has a principal subbundle with reductive structure group if and only if its pullback $f^*P$ along $f$ does (Theorem~\ref{t:redsub} below). 
Suppose further that $H^0(X \, , \, \sO_X)$ and $H^0(X' \, , \, \sO_{X'})$ are henselian local with residue field $k$. 
In view of the above, the question of pullback along $f$ of isomorphism classes of principal $G_1$\nd subbundles of $P$ for $G_1$ reductive 
reduces to that of the pullback of minimal principal bundles. 
Such a pullback need not be minimal: any principal bundle 
with finite structure group for example is trivialised by pullback along a finite \'etale cover.
It is however always almost minimal (Theorem~\ref{t:almin} below).
It can be shown that $f$ factors as a proper surjective morphism $X' \,\longrightarrow\, X_1$
which induces a surjection on fundamental groups, followed by a finite \'etale morphism $X_1 \,\longrightarrow\, X$
(see the paragraph following Theorem~\ref{t:redsub} below).
If we suppose further that $X_1 \,=\, X$, or equivalently that $f^*Z$ is connected for every connected \'etale cover $Z$ of $X$,
then $f^*$ preserves minimal principal bundles, and indeed for every reductive $k$\nd subgroup $G_1$ of $G$ it induces a bijection 
from isomorphism classes of principal $G_1$\nd subbbundles of $P$ to those of $f^*P$ (Theorem~\ref{t:pull} below).
As one consequence, if $G$ is reductive, then two principal $G$\nd bundles $P_1$ and $P_2$ over $X$ are isomorphic if and only if $f^*P_1$ and $f^*P_2$ are isomorphic (Corollary~\ref{c:pull} below).

Consider now the complex analytic case.
Given a complex Lie group $J$, we define principal $J$\nd bundles over a complex analytic space,
principal $J_1$\nd subbundles for a closed complex Lie subgroup $J_1$ of $J$, and isomorphism of such principal $J_1$\nd subbundles, 
similarly to the algebraic case.
We confine attention to those complex Lie groups $J$ which are algebraic, in the sense that $J$ is the complex Lie 
group $G_\mathrm{an}$ associated to an affine algebraic group $G$ over $\C$.
If there exists such a $G$ which is reductive, we say that $J$ is reductive.
For $J$ reductive, minimal and almost minimal principal $J$\nd bundles over a complex analytic space are defined as in the algebraic case.
For a non-empty complex analytic space $X$, the $\C$\nd algebra $H^0(X \, , \, \sO_X)$ is henselian local with residue field $\C$ 
if and only if the restriction to $X_\mathrm{red}$ of every holomorphic function on $X$ is constant \cite[Lemma~2.1]{BisO'S21}, and in 
particular when $X$ is reduced, if and only if $H^0(X \, , \, \sO_X) \,=\, \C$. 
We then have complex analytic analogues Theorems~\ref{t:pushan}, \ref{t:redsuban}, \ref{t:alminan}, 
\ref{t:pullan} and Corollary~\ref{c:pullan} respectively of Theorems~\ref{t:push}, \ref{t:redsub}, \ref{t:almin}, \ref{t:pull} and 
Corollary~\ref{c:pull}. 
The almost triviality result \cite[Lemma~4.2]{BisDum22} follows for example from Theorem~\ref{t:alminan} together with the equivalent form
Corollary~\ref{c:conjan} of Theorem~\ref{t:pushan}.

Theorem~\ref{t:push}, on which most of the results in the algebraic case in this paper depend, is a particular case of 
\cite[Corollary~13.9]{O19}, which is there deduced from a corresponding result \cite[Corollary~12.11(i)]{O19} for groupoids; see Remark 
\ref{rem-ap} for an alternative approach. Here we use \cite[Corollary~12.11(i)]{O19} directly, after first recalling the well-known 
dictionary between principal bundles and transitive affine groupoids. The other result in the algebraic case taken from \cite{O19} is 
Proposition~\ref{p:spl}, which is a particular case of \cite[Corollary~10.14]{O19}.

The analytic analogue Theorem~\ref{t:pushan} of Theorem~\ref{t:push} is proved as an application of the splitting theorem for tensor categories proved in \cite{AndKah}, \cite{O} and \cite{O19}, together with the analytic form of the
dictionary between principal bundles and tensor functors from a category of representations to a category of vector bundles.
The algebraic form of this dictionary is well-known, and Theorem~\ref{t:push} can be proved by an almost identical argument if preferred.
The proof of the analytic analogue Proposition~\ref{p:splan} of Proposition~\ref{p:spl} is self-contained, and does not depend on \cite{O19}.
Again Proposition~\ref{p:spl} can be proved by an almost identical argument.
Once Theorem~\ref{t:pushan} and Proposition~\ref{p:splan} have been established, the proofs of the analytic results are almost identical
to those of the algebraic results, and are omitted.

In both the algebraic and analytic cases, the results for pullback of principal bundles along a morphism $f\,:\,X' \,\longrightarrow\, X$
are proved in a more general form where either one of two conditions \ref{i:redsublf} or \ref{i:redsubprop} on $f$ holds 
(see Theorems~\ref{t:redsub} and \ref{t:redsuban}), where \ref{i:redsubprop} corresponds to the condition above that $f$ be proper and surjective.
In each case the required result is first proved under condition \ref{i:redsublf}.
The result under condition \ref{i:redsubprop} is then deduced from that under \ref{i:redsublf} using the lemmas given in the next section.

\section{Preliminaries}

Recall that a ringed space $X$ is said to be normal if for each point $x \,\in\, X$ the stalk 
$\sO_{X,x}$ is an integral domain which is integrally closed in its field of fractions. If $X$ 
is a ringed space with $\sO_{X,x}$ an integral domain for each $x \,\in\, X$, an $\sO_X$\nd module 
$\sF$ will be called \emph{torsion free} if $\sF_x$ is a torsion free $\sO_{X,x}$\nd module for 
every $x \,\in\, X$.

\begin{lem}\label{l:lfree}
Let $X$ be a locally noetherian normal scheme and $\sF$ a torsion free coherent $\sO_X$\nd module.
Then $X$ has an open subscheme $U$ with complement everywhere of codimension at least two
such that the restriction of $\sF$ to $U$ is locally free.
\end{lem}

\begin{proof}
The subset $U$ of those $x \,\in\, X$ for which $\sF_x$ is a free $\sO_{X,x}$\nd module
is open in $X$, and the restriction of $\sF$ to $U$ is locally free.
If $x \,\in \,X$ is of codimension at most one, then $\sO_{X,x}$ is either a field or a
discrete valuation ring, so that $x \,\in\, U$ because $\sF_x$ is torsion free.
\end{proof}

\begin{lem}\label{l:restr}
Let $X$ be a locally noetherian normal scheme, $X'$ a reduced and irreducible scheme, and
let $f\,:\,X' \,\longrightarrow\, X$ be a proper surjective morphism.
Then $X$ has an open subscheme $U$ with complement everywhere of codimension at least two
such that the restriction of $f_*\sO_{X'}$ to $U$ is a locally free $\sO_U$\nd module.
\end{lem}

\begin{proof}
Since $f$ is surjective and $X$ is reduced, the natural homomorphism $\sO_X \,\longrightarrow
\, f_*\sO_{X'}$ has trivial kernel. This shows that $f_*\sO_{X'}$ is a torsion free $\sO_X$\nd module, because
each stalk of it is an integral domain. The required result now follows from Lemma~\ref{l:lfree}.
\end{proof}

Let $R$ be a commutative ring, and let $M$ be an $R$\nd module of finite presentation. If
\begin{equation*}
R^n \,\longrightarrow\, R^m \,\longrightarrow\, M \,\longrightarrow\, 0
\end{equation*}
is a finite presentation of $M$, recall that for each integer $r \,\ge\, -1$, we have an ideal of $R$
which is generated by the $(m-r) \times
(m-r)$ minors of the $m \times n$ matrix defining $R^n\,\longrightarrow\, R^m$ if 
$0\,\le \, r\,< \,m$, and it is defined to be $0$ when $r = -1$ and $R$ when $r \,\ge\, m$; this ideal of $R$
for integer $r$ is actually independent of the choice of presentation
for each $r$ \cite[\href{https://stacks.math.columbia.edu/tag/07Z8}{Tag 07Z8}]{stacks-project}. 
It is a finitely generated ideal of $R$, which is the $r$th Fitting ideal $\mathrm{Fitt}_r(M)$ of $M$, and its formation
commutes with extension along homomorphisms $R \,\longrightarrow\, R'$ of commutative rings. We have
\begin{equation*}
\mathrm{Fitt}_r(M) \,\,\subset\,\, \mathrm{Fitt}_{r'}(M)
\end{equation*} 
for $r \,\le\, r'$.
If $Z_r$ is the closed subscheme of $\Spec(R)$ defined by $\mathrm{Fitt}_r(M)$, then for $r \ge 0$ a morphism of schemes from $T$ to $\Spec(R)$
factors through the subscheme $Z_{r-1} - Z_r$ if and only if the pullback to $T$ of the quasi-coherent sheaf associated to $M$ is locally free
of rank $r$ \cite[\href{https://stacks.math.columbia.edu/tag/05P8}{Tag 05P8}]{stacks-project}.

Let $X$ be a ringed space with $\sO_X$ coherent, and let $\sF$ be a coherent $\sO_X$\nd module.
Then the ideals $\mathrm{Fitt}_r(\sF_x)$ of the rings $\sO_{X,x}$ for $x \,\in\, X$ are the stalks of a coherent ideal $\mathrm{Fitt}_r(\sF)$ of $\sO_X$.
Indeed, if $\sF$ has a finite presentation on the open subset $U$ of $X$, then the restriction of $\mathrm{Fitt}_r(\sF)$ to $U$ is generated
by appropriate minors of the matrix defining a presentation.

\begin{lem}\label{l:lfreean}
Let $X$ be a normal complex analytic space, and let $\sF$ be a torsion free coherent $\sO_X$\nd module.
Then $X$ has an open subspace $U$ with complement an analytic subset everywhere of
codimension at least two such that the restriction of $\sF$ to $U$ is locally free.
\end{lem}

\begin{proof}
We may suppose that $X$ is non-empty and connected, and hence irreducible \cite[9.1.2]{GraRem84}.
Let $r$ be the smallest nonnegative integer such that 
\begin{equation*}
\mathrm{Fitt}_r(\sF) \,\ne\, 0.
\end{equation*}
For each $x \,\in\, X$, denote by $Z(x)$ the closed subscheme of $\Spec(\sO_{X,x})$ defined by
the ideal $\mathrm{Fitt}_r(\sF)_x$ of $\sO_{X,x}$,
and also denote by $U(x)$ the complement of $Z(x)$.
Then $U(x)$ is the largest open subscheme of $\Spec(\sO_{X,x})$ to which the restriction of the quasi-coherent sheaf $\widetilde{\sF_x}$ associated to $\sF_x$ is locally free of rank $r$.
Further $\mathrm{Fitt}_r(\sF)_x \,\ne\, 0$ for every $x \,\in\, X$, because $X$ is reduced and irreducible \cite[9.1.1]{GraRem84}.
Consequently, $U(x)$ is dense in $\Spec(\sO_{X,x})$ for each $x$, because $\Spec(\sO_{X,x})$ is reduced and irreducible.
It follows that $U(x)$ is the largest open set of $\Spec(\sO_{X,x})$
such that the restriction of $\widetilde{\sF_x}$ to it is locally free.
Since $\Spec(\sO_{X,x})$ is normal and $\widetilde{\sF_x}$ is torsion-free, Lemma~\ref{l:lfree} thus shows that $Z(x)$ is
of codimension at least two for each $x$.

Denote by $Z$ the closed analytic subspace of $X$ defined by $\mathrm{Fitt}_r(\sF)$, and by $U$ the complement of $Z$.
Then $Z$ is of codimension at least two at each $z\,\in\, Z$, because $Z(z)$ is of codimension
at least two in $\Spec(\sO_{X,z})$. Therefore, $U$ has the required
properties, because $x\,\in\, U$ if and only if $U(x) \,= \,\Spec(\sO_{X,z})$ and
$U(x) \,=\,\Spec(\sO_{X,z})$ if and only if $\sF_x$ is a free $\sO_x$\nd module.
\end{proof}

\begin{lem}\label{l:restran}
Let $X$ be a normal complex analytic space, $X'$ a reduced and irreducible complex analytic space, 
and let $f\,:\,X'\,\longrightarrow\, X$ be a proper surjective morphism.
Then $X$ has an open subspace $U$ with complement an analytic subset everywhere
of codimension at least two,
such that the restriction of $f_*\sO_{X'}$ to $U$ is a locally free $\sO_U$\nd module.
\end{lem}

\begin{proof}
Let $X'\,\longrightarrow\, X_1 \,\xrightarrow{\,\,f_1\,}\, X$ be the Stein factorisation of
$f$. Then $f_1$ is finite with $f_1{}_*\sO_{X_1}\,\longrightarrow\, f_*\sO_{X'}$ an isomorphism.
Replacing $X'$ by $X_1$, we may thus suppose that $f$ is finite.
Denote by $X''$ the normalisation of $X'$ and by $f'\,:\,X'' \,\longrightarrow\, X$
the composite of $f$ with the natural projection $X'' \,\longrightarrow\, X'$.
Then $f'$ is finite and surjective, and we have a monomorphism of $\sO_X$\nd modules
\begin{equation*}
0 \,\longrightarrow\, f_*{\sO_{X'}} \,\longrightarrow\, f'{}\!_*{\sO_{X''}}.
\end{equation*}
For any $x \,\in\, X$ the inverse images under $f'$ of the neighbourhoods of $x$ form a base
for the neighbourhoods of the fibre
$\{x_1,\,x_2,\, \ldots ,\,x_n\}$ of $f'$ above $x$. 
The stalk of $f'{}\!_*{\sO_{X''}}$ at $x$ is thus given by 
\begin{equation*}
(f'{}\!_*{\sO_{X''}})_x \,=\, \sO_{X'',x_1} \oplus \sO_{X'',x_2} \oplus \cdots \oplus \sO_{X'',x_n},
\end{equation*}
with each $\sO_{X'',x_i}$ a finite $\sO_{X,x}$\nd algebra. Since $X''$ is irreducible and $f'$ is finite
and surjective, both $X$ and $X''$ are equidimensional (see \cite[9.1.3]{GraRem84}) of 
the same dimension \cite[5.4.1]{GraRem84}. Therefore, $f'$ is open \cite[5.4.3]{GraRem84}, so that each $\sO_{X,x} \,
\longrightarrow\,\sO_{X'',x_i}$ is an injective homomorphism of integral domains.
It follows that $f'{}\!_*{\sO_{X''}}$ and hence $f_*{\sO_{X'}}$ is a torsion-free $\sO_X$\nd module.
The required result now follows from Lemma~\ref{l:lfreean}.
\end{proof}

See for example \cite[6.45]{GorWed10} for Lemma~\ref{l:iso} and \cite[7.4.2]{GraRem84} 
for Lemma~\ref{l:isoan}.

\begin{lem}\label{l:iso}
Let $X$ be a locally noetherian normal scheme, and let $j\,:\,U\,\longrightarrow\, X$ be the
embedding of an open subscheme whose complement is everywhere of codimension at least two. 
Then the natural homomorphism $\sO_X \,\longrightarrow\, j_*\sO_U$ is an isomorphism.
\end{lem}

\begin{lem}\label{l:isoan}
Let $X$ be a normal complex analytic space, and let $j\,:\,U\,\longrightarrow\, X$
be the embedding of an open subspace whose complement is an analytic subset everywhere of codimension at least two.
Then $\sO_X \,\longrightarrow\, j_*\sO_U$ is an isomorphism.
\end{lem}

Let $X$ be a ringed space and $\sF$ an $\sO_X$\nd module.
The cohomology group $H^1(X,\,\sF)$ may be identified with the set $\Ext_{\sO_X}^1(\sO_X,\,
\sF)$ of isomorphism classes of extensions of $\sO_X$ by $\sF$.
The pullback homomorphism
\begin{equation*}
i_{\sF}\,:\,H^1(X,\,\sF) \,\longrightarrow\, H^1(X',\,f^*\sF)
\end{equation*}
along a morphism $f\,:\,X'\,\longrightarrow\, X$ of ringed spaces is then
the pullback of extensions. If 
\begin{equation*}
j_{\sF'}\,:\,H^1(X,\,f_*\sF')\,\longrightarrow\, H^1(X',\,\sF')
\end{equation*}
is $i_{f_*\sF'}\,:\,H^1(X,\,f_*\sF') \,\longrightarrow\, H^1(X',\,f^*f_*\sF')$
followed by the homomorphism of cohomologies given by the
push forward along the counit $f^*f_*\sF'\,\longrightarrow\, \sF'$,
then the naturality of $i_{\sF}$ and the triangular identity for $f_*$ and $f^*$
together give a factorisation
\begin{equation}\label{e:ifac}
H^1(X,\,\sF)\,\longrightarrow\, H^1(X,\,f_*f^*\sF) \,\xrightarrow{\,\,j_{f^*\sF}\,}
\, H^1(X',\,f^*\sF)
\end{equation} 
of $i_{\sF}$, with the first arrow being the one given by the push forward
along the unit $\sF\,\longrightarrow\, f_*f^*\sF$.
Further $j_{\sF'}$ is injective for every $\sF'$ for the following reason: If $\sE$
is an extension of $\sO_X$ by $f_*\sF'$, and $\sE'$ is the push forward of $f^*\sE$
along $f^*f_*\sF'\,\longrightarrow\, \sF'$, then a splitting $\sE' \,\longrightarrow\, \sF'$ 
of $\sE'$ defines by its composite $f^*\sE \,\longrightarrow\, \sF'$ with
$f^*\sE \,\longrightarrow\, \sE'$ and adjunction a splitting
$\sE \,\longrightarrow\, f_*\sF'$ of $\sE$. 

\begin{lem}\label{l:inj}
Let $f\,:\,X' \,\longrightarrow\, X$ be a morphism of ringed spaces such that $f_*\sO_{X'}$
is a locally free $\sO_X$\nd module of finite type.
Suppose that the rank of $f_*\sO_{X'}$ is invertible everywhere on $X$.
Then for every locally free $\sO_X$\nd module of finite type $\sW$, the pullback
homomorphism $$i_{\sW}\, :\, H^1(X,\,\sW)\, \longrightarrow\,H^1(X',\,f^*\sW)$$ is injective.
\end{lem}

\begin{proof}
Since $j_{f^*\sW}$ is injective, it is enough, after factoring $i_{\sW}$ as in
\eqref{e:ifac}, to show that the unit
\begin{equation*}
\eta_{\sW}\,:\sW \,\longrightarrow\, f_*f^*\sW
\end{equation*}
has a left inverse.

The push forward $f_*f^*\sW$ of $f^*\sW$ has a structure of $f_*\sO_{X'}$\nd module. 
Consequently, $\eta_{\sW}$ factors as
\begin{equation}\label{e1}
\sW \,\xrightarrow{\,\,u \otimes_{\sO_X} \sW\,\,}\, f_*\sO_{X'} \otimes_{\sO_X} \sW
\,\longrightarrow\, f_*f^*\sW,
\end{equation}
where $u\,:\,\sO_X \,\longrightarrow\, f_*\sO_{X'}$ is
the identity of the $\sO_X$\nd algebra $f_*\sO_{X'}$ and the second arrow is a morphism of $f_*\sO_{X'}$\nd modules.
Arguing locally over $X$ shows that the second arrow in \eqref{e1} is an isomorphism.
Since $f_*\sO_{X'}$ is locally free of finite type, we have a trace morphism
$f_*\sO_{X'} \,\longrightarrow\, \sO_X$ of $\sO_X$\nd modules. Its composite with
$u$ is the endomorphism of $\sO_X$ given by the rank $r$ of $f_*\sO_{X'}$.
Since $r$ is invertible, it follows that $u$ and hence $u \otimes_{\sO_X} \sW$ has a left inverse.
\end{proof}

\section{Schemes over an algebraically closed field of characteristic 0}

\emph{Throughout this section, $k$ is an algebraically closed field of characteristic $0$.}

Let $G$ be an affine algebraic $k$\nd group.
By a principal $G$\nd bundle over a $k$\nd scheme $X$ is meant a scheme $P$ over $X$ together with a right action of $G$ on $P$ above $X$,
such that locally in the \'etale topology, $P$ is trivial, i.e.,\ isomorphic to $X \times_k G$ with $G$ acting by right translation.
A morphism of principal $G$\nd bundles over $X$ is a morphism of their underlying schemes over $X$ which is compatible with 
the right actions of $G$.
Every such morphism is an isomorphism.
The pullback $f^*P$ of $P$ along a morphism of $k$\nd schemes $f\,:\,X'
\,\longrightarrow\, X$, with its canonical right action of $G$ over $X'$, is a principal 
$G$\nd bundle over $X'$.

Given a $k$\nd homomorphism $h\,:\,G \,\longrightarrow\, G'$ of affine algebraic $k$\nd groups and a principal $G$\nd bundle $P$ 
over $X$, there exists a principal $G'$\nd bundle $P'$ over $X$ and a morphism $l\,:\,P \,\longrightarrow\, P'$ 
over $X$ such that
\begin{equation*}
l(pg) \,= \,l(p)h(g)
\end{equation*}
for points $p$ of $P$ and $g$ of $G$, and the pair $(P',\,l)$ 
is unique up to a unique isomorphism. 
The pair $(P',\,l)$, or simply the principal $G'$\nd bundle $P'$ is said to be the push forward of $P$ along $h$.
For any $g' \,\in \,G'(k)$, if $h'\,:\,G \,\longrightarrow\, G'$ is the conjugate morphism defined by $g\, \longmapsto\, g' h(g) g'{}^{-1}$,
then $(P',\,l')$, where $l'$ is the right translate of $l$ by $g'{}^{-1}$, is the push forward of $P$ along $h'$.

Let $G_1$ be an affine $k$\nd subgroup of $G$.
A \emph{principal $G_1$\nd subbundle of $P$} is defined as a closed subscheme $P_1$ of $P$
such that the action of $G$ on $P$ restricts to an action of $G_1$ on $P_1$ with $P_1$ a principal $G_1$\nd bundle over $X$.
The push forward of $P_1$ along the inclusion map $G_1 \, \hookrightarrow\, G$ is then $P$.
Two principal $G_1$\nd subbundles of $P$ will be called isomorphic if they are isomorphic as principal 
$G_1$\nd bundles over $X$.
By uniqueness of push forward,
any such isomorphism is the restriction of a unique automorphism of the principal $G$\nd bundle $P$ over $X$. 

We do not require reductive algebraic $k$\nd groups to be connected.
Let $G_1$ be a reductive $k$\nd subgroup of $G$. 
Then the quotient $G/G_1$ is affine. 
If $P$ is a principal $G$\nd bundle over a $k$\nd scheme $X$, then the 
quotient $P/G_1$ exists, and \'etale locally over $X$ is isomorphic to to $X \times_k (G/G_1)$. 
In particular $P/G_1$ is affine over $X$. 
Further $P$ is a principal $G_1$\nd bundle over $P/G_1$, and it is a
principal $G_1$\nd subbundle of the pullback, to $P/G_1$, of the principal $G$\nd bundle $P$.
Given a cross-section of $P/G_1$ over $X$ we may pull back, to $X$, this
principal $G_1$\nd bundle. This produces a bijection from the set of cross-sections of 
$P/G_1$ over $X$ to the set of principal $G_1$\nd subbundles of $P$ over $X$.
Formation of $P/G_1$ commutes with pullback.

The group-scheme $\underline{\Aut}_G(P)$ over $X$ of automorphisms of a principal $G$\nd bundle $P$ over $X$ similarly exists, and \'etale 
locally over $X$ is isomorphic to $X \times_k G$.

Proposition~\ref{p:bij} below is an almost immediate consequence of the definitions and, 
with appropriate definitions, holds with $k$ an arbitrary field. 
By contrast, results such as Theorems~\ref{t:pull} and \ref{t:pullan} below --- with conditions on $X$ and $X'$, 
but a less stringent condition on $f$ --- are more difficult, and apply only to isomorphism classes 
of subbundles.

\begin{prop}\label{p:bij}
Let $f\,:\,X'\,\longrightarrow\, X$ be a morphism of $k$\nd schemes with $\sO_X
\,\longrightarrow\, f_*\sO_{X'}$ an isomorphism. Let $G$ be an affine algebraic $k$\nd group
and $G_1\, \subset\, G$ a reductive $k$\nd subgroup.
Then for any principal $G$\nd bundle $P$ over $X$, pullback along the projection of $f^*P$ onto $P$ defines a bijection
from the set of (respectively, the set of isomorphism classes of) principal $G_1$\nd subbundles
of $P$ to the set of (respectively, the set of isomorphism classes of) principal $G_1$\nd subbundles of $f^*P$.
\end{prop}

\begin{proof}
By the above, it suffices to show that for any scheme $Z$ affine over $X$, pullback along $f$
defines a bijection from the cross-sections of $Z$ to the cross-sections of $f^*Z$.
Indeed, taking $Z \,=\, P/G_1$ will give the required result for sets of subbundles, and
taking $Z \,=\, \underline{\Aut}_G(P)$ will then give that 
for isomorphism classes.

We have $Z \,=\, \Spec(\sR)$ for a quasi-coherent commutative $\sO_X$\nd algebra $\sR$.
A cross-section of $Z$ may be identified with a morphism $\sR
\,\longrightarrow\, \sO_X$ of $\sO_X$\nd algebras, and a cross-section of 
$f^*Z$ with a morphism $\sR\,\longrightarrow\, f_*\sO_{X'}$ of $\sO_X$\nd algebras.
Pullback of cross-sections is then given by composition with $\sO_X
\,\longrightarrow\, f_*\sO_{X'}$.
\end{proof}

Theorem~\ref{t:push} below will be deduced from \cite[Corollary~12.11(i)]{O19} using the dictionary between
principal bundles and transitive affine groupoids, which we recall next.
An alternative way of proving Theorem~\ref{t:push} is described in Remark \ref{rem-ap}.

Let $X$ be a $k$\nd scheme.
Recall that a groupoid over $X$ is a $k$\nd scheme $K$ with a source $k$\nd morphism $d_1$ and a target $k$\nd morphism $d_0$ from
$K$ to $X$, together with an identity $X \,\longrightarrow\, K$ and a composition 
\begin{equation*}
\circ\,:\,K \times_{{}^{d_1}X^{d_0}} K\,\longrightarrow\, K
\end{equation*} 
which is associative, has $X \,\longrightarrow\, K$ as a left and right identity, and has inverses.
The points of $X$ and $K$ in a given $k$\nd scheme then form respectively the objects and
arrows of a groupoid, i.e.,\ a category in which 
every morphism is an isomorphism.
A morphism $K \,\longrightarrow\, K'$ of groupoids over $X$ is a morphism of $k$\nd schemes which is compatible with the source, target, identity and composition
of $K$ and $K'$.
A subgroupoid of a groupoid $K$ over $X$ is a closed subscheme $K'$ of $K$ such that the groupoid structure of $K$ induces one on $K'$.

The source and target morphisms $d_1$ and $d_0$ of a groupoid $K$ over $X$ are the components of a morphism
\begin{equation}\label{d01}
(d_0,\,d_1)\,:\,K \,\longrightarrow\, X \times_k X
\end{equation}
over $k$. The inverse image under $(d_0,\,d_1)$ of the diagonal $X$ of $X \times_k X$ is
then a group scheme $K^\mathrm{diag}$ over $X$.
Any cross-section $v$ of $K^\mathrm{diag}$ over the diagonal $X$ of $X \times_k X$
induces an automorphism of $K$ over $X$ by conjugation:
\begin{equation}\label{ec}
K \,\,\iso\,\, K\, ,\ \ \, w\, \longmapsto\, v(d_0(w))\circ w \circ v(d_1(w))^{-1}.
\end{equation}

A groupoid over $X$ is said to be affine (respectively, of finite type) if the morphism
$(d_0,\,d_1)$ in \eqref{d01} is affine (respectively, of finite type).
For $X$ non-empty, a groupoid over $X$ which is affine and of finite type is said to be transitive if $(d_0,\, d_1)$ is surjective and smooth.

The principal bundles over $X$ form a category whose objects are pairs $(G,\,P)$ with $G$ an affine algebraic $k$\nd group and $P$ 
a principal $G$\nd bundle over $X$,
where a morphism from $(G,\,P)$ to $(G',\,P')$ is a pair $(h,\,l)$ (which is in fact
completely determined by $l$) with $h$ a $k$\nd homomorphism 
$G \,\longrightarrow\, G'$ and $l$ a morphism $P \,\longrightarrow\, P'$ over $X$ such that $(P',\,l)$ is the push forward of $P$ along $h$.
When $X$ is non-empty, we may
define as follows a functor $\underline{\Iso}_-(-)$ from this category to transitive affine groupoids of finite type over $X$.
The points with source $x_1$ and target $x_0$ in a $k$\nd scheme $Z$ of the groupoid
$\underline{\Iso}_G(P)$ over $X$ are
the isomorphisms $P_{x_1} \,\iso\, P_{x_0}$ of principal $G$\nd bundles over $Z$.
The morphism $\underline{\Iso}_h(l)$ from $\underline{\Iso}_G(P)$ to $\underline{\Iso}_{G'}(P')$
sends the point $v\,:\,P_{x_1} \,\iso\, P_{x_0}$ of 
$\underline{\Iso}_G(P)$ over $(x_0,\,x_1)$ to the unique point
$$v'\,:\,P'{}\!_{x_1} \,\iso\, P'{}\!_{x_0}$$ of $\underline{\Iso}_{G'}(P')$
over $(x_0,\,x_1)$ such that
\begin{equation*}
v'(l(p))\,\, =\,\, l(v(p))
\end{equation*}
for every point $p$ of $P$ over $x_1$. 

The connection between principal bundles and transitive affine groupoids over $X$ is 
easiest to describe when $X$ has a $k$\nd point $x$. We may then consider the category of 
pairs $(G,\,P)$ equipped with a $k$\nd point of $P$ above $x$, where the morphisms are 
those preserving $k$\nd points above $x$. The functor $\underline{\Iso}_-(-)$ defines an 
equivalence from this category to transitive affine groupoids of finite type over $X$, with 
a quasi-inverse which sends $K$ to $(G,\,P)$ with $G$ the fibre of $K$ above $(x,\,x)$ and 
$P$ the inverse image in $K$ of $X \times x$, where the $k$\nd point above $x$ is the 
identity of $P_x \,=\, G$ and the right action of $G$ on $P$ is by composition.

To describe what happens for arbitrary non-empty $X$, call two morphisms $(h,\,l)$ and $(h',\,l')$ from
$(G,\,P)$ to $(G',\,P')$ conjugate if there exists a $g' \,\in\, G'(k)$ (necessarily unique) such that $h'(g)\,= \, g'h(g)g'{}^{-1}$
and $l'(p) \,=\, l(p)g'{}^{-1}$ for points $g$ of $G$ and $p$ of $P$.
Then $\underline{\Iso}_-(-)$ defines an equivalence from pairs $(G,\,P)$ and conjugacy classes of morphisms
to transitive affine groupoids of finite type over $X$. 
Indeed the full faithfulness can be seen from the fact that the 
$(h,\,l)$ with $\underline{\Iso}_h(l)$ a given morphism from $\underline{\Iso}_G(P)$ to $\underline{\Iso}_{G'}(P')$ are the $k$\nd points of a
$k$\nd scheme which is simply transitive under the action of $G$ by conjugation, which follows by faithfully flat descent from any extension $k'$ 
of $k$ for which $X$ has a $k'$\nd point.
The essential surjectivity can be proved using the following condition \cite[Lemma~5.1 and the paragraph preceding it]{O19}:
a transitive affine groupoid of finite type $K$ over $X$ is of the form $\underline{\Iso}_G(P)$ for some $(G,\,P)$ if and only if a simply transitive
$K$\nd scheme \cite[p.~21]{O19} exists.
Since $k$ is algebraically closed, this condition is satisfied by \cite[Lemma~7.3]{O19}.

The cross-sections of the group scheme $\underline{\Iso}_G(P)^\mathrm{diag}$
over the diagonal in $X\times_k X$ are the automorphisms of the principal 
$G$\nd bundle $P$ over $X$.
Conjugation by such a cross-section $v$ is the automorphism
$\underline{\Iso}_{1_G}(v)$ of $\underline{\Iso}_G(P)$ (see \eqref{ec}).

The transitive affine subgroupoids of the groupoid $\underline{\Iso}_G(P)$ over $X$ are all of the
form $\underline{\Iso}_{G'}(P')$ for $G'$ a $k$\nd subgroup of 
$G$ and $P'$ a principal $G'$\nd subbundle of $P$.
Further $\underline{\Iso}_{G'}(P') = \underline{\Iso}_G(P)$ if and only if $G' = G$.

Let $K$ be a transitive affine groupoid of finite type $K$ over $X$.
Then $K^\mathrm{diag}$ is a smooth group scheme of finite type over $X$.
As in \cite{O19}, call $K$ \emph{reductive} if the fibres of $K^\mathrm{diag}$ are reductive, and \emph{minimally reductive} 
if further $K'\,=\, K$ for every reductive transitive affine subgroupoid $K'$ of $K$.

Let $G$ be a reductive algebraic $k$\nd group and $P$ a principal $G$\nd bundle over a $k$\nd 
scheme $X$. We say that $P$ is \emph{minimal} if a principal $G'$\nd subbundle of $P$ exists 
for a reductive $k$\nd subgroup $G'$ of $G$ only if $G'\, =\, G$. The groupoid 
$\underline{\Iso}_G(P)$ over $X$ is reductive if and only if $G$ is reductive. When this is so, 
$\underline{\Iso}_G(P)$ is minimally reductive if and only if $P$ is minimal.
Any push forward of a minimal principal $G$\nd bundle over $X$ along a surjective $k$\nd homomorphism is minimal. 

Theorem~\ref{t:push} below is the case where $H \,=\, X$,
$k \,=\, \overline{k}$, and $F$ and $F'$ are of finite type of \cite[Corollary~13.9]{O19}.
It is equivalent to the case where $k$ is algebraically closed, 
$H\,=\,X$, and $K$ is of finite type of \cite[Corollary~12.11(i)]{O19}, from which we deduce it here.

\begin{remark}\label{rem-ap}
Theorem~\ref{t:push} can also be proved in a similar way to its analytic analogue Theorem~\ref{t:pushan} below.
To do this, it is first necessary to prove the algebraic analogue of Theorem~\ref{t:min}.
This can be done by replacing Lemma~\ref{l:Pequ} in the proof of Theorem~\ref{t:min} by its algebraic analogue,
which is well known from the theory of Tannakian categories.
The algebraic analogue of Theorem~\ref{t:min} is equivalent to the case where $k$ is algebraically closed and $H \,=
\,K$ of the minimal reductive criterion of \cite[Theorem~12.5(ii)]{O19}.
Theorem~\ref{t:push} can be deduced from it in the same way as Theorem~\ref{t:pushan} from Theorem~\ref{t:min}.
\end{remark}

\begin{thm}\label{t:push}
Let $X$ be a $k$\nd scheme, and let $G_0$ and $G$ be affine algebraic $k$\nd groups with $G_0$
reductive. Let $P_0$ be a minimal principal $G_0$\nd bundle over $X$, and
let $h_1$ and $h_2$ be $k$\nd homomorphisms from $G_0$ to $G$.
Suppose that $H^0(X,\,\sO_X)$ is a henselian local $k$\nd algebra with residue field $k$.
Then the push forwards of $P_0$ along $h_1$ and $h_2$ are isomorphic if
and only if $h_1$ and $h_2$ are conjugate.
\end{thm}

\begin{proof}
The ``if'' part has been seen above, even without any condition on $G_0$ or $X$.

Conversely suppose that the push forwards of $P_0$ along $h_1$ and $h_2$ are isomorphic. Then 
for some principal $G$\nd bundle $P$ over $X$ there exist
$l_1,\, l_2\, :\, P_0\, \longrightarrow\, P$ such that $(h_1,\,l_1)$ 
and $(h_2,\,l_2)$ are morphisms from $(G_0,P_0)$ to $(G,P)$.
By \cite[Corollary~12.11(i)]{O19}, $\underline{\Iso}_{h_2}(l_2)$ is the conjugate of $\underline{\Iso}_{h_1}(l_1)$
from $\underline{\Iso}_{G_0}(P_0)$ to $\underline{\Iso}_G(P)$ by a cross-section $v$ of 
$\underline{\Iso}_G(P)^\mathrm{diag}$.
Then
\begin{equation*}
\underline{\Iso}_{h_2}(l_2) = \underline{\Iso}_{1_G}(v) \circ \underline{\Iso}_{h_1}(l_1) = 
\underline{\Iso}_{h_1}(v \circ l_1).
\end{equation*}
Thus $(h_1,\, v \circ l_1)$ and $(h_2,\,l_2)$ and hence $h_1$ and $h_2$ are conjugate.
\end{proof}

Let $G$ be a finite $k$\nd group.
Then any principal $G$\nd bundle $P$ over $X$ is finite \'etale over $X$, 
and any principal subbundle of $P$ is an open and closed subscheme of $P$.

Suppose that $G$ is finite and $X$ is connected. Then $P$ is the disjoint union of at most 
$d$ connected components, where $d$ is the order of $G$. Any connected component $P_0$ of 
$P$ is a principal $G_0$\nd subbundle of $P$, where $G_0$ is the subgroup of $G$ consisting 
of those points that send $P_0$ into itself under the right action on $P$. It follows that 
$P$ is minimal if and only it is connected. Further if $G_i$ is a subgroup of $G$, and 
$P_i$ is a principal $G_i$\nd subbundle of $P$ for $i \,=\, 0,\,1$, with $P_0\, \subset\, 
P$ minimal, then there exists an element $g \,\in\, G(k)$ such that $G_1$ contains 
$gG_0g^{-1}$ and $P_1$ contains $P_0g^{-1}$.

Corollary~\ref{c:conj} below shows that a similar conjugacy result holds for an arbitrary affine 
algebraic $k$\nd group $G$, provided that we require the $k$\nd subgroups $G_i$ to be 
reductive, impose a more stringent condition on $X$, and require conjugacy only up to 
isomorphism.
The case where $H^0(X,\,\sO_X) \, = \, k$ of Corollary~\ref{c:conj} is a result of Bogomolov \cite[p.~401, Theorem~2.1]{Bo94}.

\begin{cor}\label{c:conj}
Let $X$ be a $k$\nd scheme, $G$ an affine algebraic $k$\nd group and $P$
a principal $G$\nd bundle over $X$.
For $i\,=\, 0,\,1$, let $G_i$ be a reductive $k$\nd subgroup of $G$ and $P_i$
a principal $G_i$\nd subbundle of $P$, such that $P_0$ is minimal.
Suppose that $H^0(X,\,\sO_X)$ is a henselian local $k$\nd algebra with residue field $k$.
Then there exists an element $g \,\in \,G(k)$ such that $G_1$ contains $gG_0g^{-1}$ and $P_1$ contains a principal $gG_0g^{-1}$\nd subbundle
isomorphic to $P_0g^{-1}$. 
\end{cor}

\begin{proof}
Let $G'$ be a reductive $k$\nd subgroup of $G_0 \times_k G_1$ for which $P_0 \times_X P_1$ has a minimal principal $G'$\nd subbundle $P'$.
For $i\,=\, 0,\,1$, write $h_i\,:\,G'\,\longrightarrow\,G_i$ for the restriction of the projection $G_0 \times_k G_1\,\longrightarrow\, G_i$
to $G'$. Then $P_i$ is the push forward of $P'$ along $h_i$.
Since $P_0$ is minimal, $h_0$ is surjective.

If $e_i\,:\,G_i \, \longrightarrow \, G$ is the embedding, then $P$ is the push forward of $P'$ along $e_i \circ h_i$ for $i\,=\, 0,\,1$.
By Theorem~\ref{t:push}, $e_1 \circ h_1$ is thus the conjugate of $e_0 \circ h_0$ by some $g \,\in \,G(k)$.
Then $G_1$ contains $gG_0g^{-1}$, and $h_1$ factors as
\begin{equation*}
G' \,\xrightarrow{\,\,h_0\,\,}\, G_0 \,\longrightarrow\, gG_0g^{-1} \,\longrightarrow\, G_1.
\end{equation*}
It follows that $P_1$ is the push forward of $P_0g^{-1}$ along the embedding of $gG_0g^{-1}$ into $G_1$.
Thus $g$ has the required properties.
\end{proof}

Let $P$ be a principal $G$\nd bundle over $X$. We have an action of $G$ on $P$ with $gp \,=\, 
pg^{-1}$ for points $g$ of $G$ and $p$ of $P$. If $V$ is a representation of $G$, arguing 
\'etale locally over $X$ shows that a pair consisting of an $\sO_X$\nd module $\sV$ and a 
$G$\nd equivariant isomorphism from the pullback of $V$ to $P$ to the pullback of $\sV$ to $P$ exists, 
and is unique up to a unique isomorphism. We write this $\sV$ as
\begin{equation*}
P \times_k^G V.
\end{equation*} 
It is the usual $\sO_X$\nd module associated to $V$ by identifying the points $(pg,\,v)$ and 
$(p,\,gv)$ of $P \times_k V$. If $V$ is finite-dimensional of dimension $n$, then $P \times_k^G 
V$ is a vector bundle, on $X$, everywhere of rank $n$. Formation of $P \times_k^G V$ is 
functorial in $P$, $G$ and $V$ and is compatible with pullback.
Similarly we define
\begin{equation*}
P \times_k^G Z
\end{equation*} 
for $Z$ an affine $G$\nd scheme.
Then for example $P/G'$ for $G'$a reductive $k$\nd subgroup of $G$ is given by taking $Z = G/G'$, 
and $\underline{\Aut}_G(P)$ by taking $Z = G_\mathrm{conj}$, the $k$\nd group $G$ with $G$ acting by conjugation.

To every vector bundle $\sV$ over $X$ of rank $n$ is associated the principal ${\rm GL}_n$\nd bundle
\begin{equation*}
\underline{\Iso}_X(\sO_X^n,\,\sV)
\end{equation*}
over $X$ of isomorphisms from the constant vector bundle $\sO_X^n$ of rank $n$ to $\sV$. 
We then have a functor $\underline{\Iso}_X(\sO_X^n,\,-)$ from vector bundles over $X$ of
rank $n$ and isomorphisms between them
to principal ${\rm GL}_n$\nd bundles over $X$. It is an equivalence, with quasi-inverse
\begin{equation*}
- \times_k^{{\rm GL}_n} k^n,
\end{equation*}
where $k^n$ is the standard $n$\nd dimensional representation of ${\rm GL}_n$.
By passing to the subbundle of those isomorphisms that respect the standard direct sum decomposition $\sO_X^r \oplus \sO_X^s$ of $\sO_X^n$ 
and a given direct sum decomposition $\sV' \oplus \sV''$ of $\sV$ with $\sV'$ of rank $r$ and $\sV''$ of rank $s$, we obtain a bijection
from such decompositions of $\sV$ to reductions of the structure group of
$\underline{\Iso}_X(\sO_X^n,\,\sV)$ from ${\rm GL}_n$ to ${\rm GL}_r\times_k{\rm GL}_s$. 

Let $P$ be a principal $G$\nd bundle over $X$.
If $V$ is a finite-dimensional representation of $G$,
and $P'$ is the push forward of $P$ along the $k$\nd homomorphism $G
\,\longrightarrow\,{\rm GL}_n$ corresponding to $V$ after choosing a basis of it,
we have an isomorphism
\begin{equation*}
P \times_k^G V \,\iso\, P' \times_k^{{\rm GL}_n} k^n.
\end{equation*}
Thus $P\times_k^G V_1$ and $P \times_k^G V_2$ are isomorphic if and only if the push
forwards of $P$ along the $k$\nd homomorphisms
$G\,\longrightarrow\, {\rm GL}_n$ corresponding to $V_1$ and $V_2$ are isomorphic.

\begin{cor}\label{c:assoc}
Let $X$ be a $k$\nd scheme, $G$ a reductive algebraic $k$\nd group and $P$
a minimal principal $G$\nd bundle over $X$.
Suppose that $H^0(X,\,\sO_X)$ is a henselian local $k$\nd algebra with residue field $k$.
Then the following two hold:
\begin{enumerate}
\item\label{i:associndec}
For any finite-dimensional representation $V$ of $G$, the vector bundle $P \times_k^G V$ over $X$ is indecomposable if and only if 
$V$ is irreducible.

\item\label{i:associso}
For any two finite-dimensional representations $V_1$ and $V_2$ of $G$, the vector bundles
$P \times_k^G V_1$ and $P \times_k^G V_2$ over $X$ 
are isomorphic if and only if $V_1$ and $V_2$ are.
\end{enumerate}
\end{cor}

\begin{proof}
\ref{i:associndec}\,
If $V$ is defined by $\rho\,:\,G \,\longrightarrow\,{\rm GL}_n$, apply Corollary~\ref{c:conj}
with ${\rm GL}_n$, $\rho(G)$, and ${\rm GL}_r \times_k {\rm GL}_s$
for $G$, $G_0$, and $G_1$ respectively. Furthermore, set the push forwards of
$P$ along $\rho$ and $G \,\longrightarrow\, \rho(G)$ in places of $P$ and $P_0$ respectively,
and also set a reduction of $P$ to ${\rm GL}_r\times_k {\rm GL}_s$ in place of $P_1$. Then
Corollary~\ref{c:conj} proves \ref{i:associndec}.

\ref{i:associso}\,
Apply Theorem~\ref{t:push} to push forwards of $P$ along the $k$\nd homomorphisms 
$G\,\longrightarrow\, {\rm GL}_n$ defining the representations $V_i$.
\end{proof}

Let $X$ be a $k$\nd scheme, $G$ an affine algebraic $k$\nd group, and $P$ a principal $G$\nd bundle over $X$ which has a principal subbundle
with reductive structure group. 
Then there exists a reductive $k$\nd subgroup $G_0$ of $G$ such that $P$ has a minimal principal $G_0$\nd subbundle $P_0$.
Suppose that $H^0(X,\,\sO_X)$ is a henselian local $k$\nd algebra with residue field $k$.
Then for each reductive $k$\nd subgroup $G_1$ of $G$, the set of isomorphism classes of principal $G_1$\nd subbundles of $P$
can be described in the following way. Denote by $T_G(G_0,\, G_1)$ the transporter from $G_0$ to $G_1$ in $G$.
Then
\begin{equation*}
T_G(G_0,\, G_1)^{-1}
\end{equation*}
is the closed subscheme of $G$ consisting of those points $t$ for which $t^{-1}G_0t$ is contained in $G_1$.
We have an action by composition on $T_G(G_0,\, G_1)^{-1}$ of the centraliser $Z_G(G_0)$ of $G_0$ in $G$ on the left and $G_1$ on the right. 
For each $k$\nd point $t$ of $T_G(G_0,\, G_1)^{-1}$, we have a principal $G_1$\nd subbundle 
\begin{equation*}
P_0tG_1
\end{equation*}
of $G$, given by the image of $P_0t \times_k G_1$ under the right action $P \times_k G \, \longrightarrow \, P$ of $G$ on $P$.
It is the unique principal $G_1$\nd subbundle of $P$ containing the principal $t^{-1}G_0t$--subbundle $P_0t$,
and it is the push forward of $P_0$ along the $k$\nd homomorphism $G_0 \, \longrightarrow \, G_1$ that sends $g_0$ to $t^{-1}g_0t$.
We have
\begin{equation*}
P_0tG_1 \,\,= \,\,P_0t'G_1
\end{equation*}
for $k$\nd points $t$ and $t'$ of $T_G(G_0,\, G_1)^{-1}$ if and only if
\begin{equation*}
t' = tg_1
\end{equation*}
for some $k$\nd point $g_1$ of $G_1$.
By Theorem \ref{t:push}, $P_0tG_1$ and $P_0t'G_1$ are isomorphic if and only if the $k$\nd homomorphisms $G_0 \, \longrightarrow \, G_1$
that send $g_0$ to $t^{-1}g_0t$ and to $t'{}^{-1}g_0t'$ are conjugate, and hence if and only if 
\begin{equation*}
t' = ztg_1
\end{equation*}
with $g_1$ a $k$\nd point of $G_1$ and $z$ a $k$\nd point of the centraliser $Z_G(G_0)$ of $G_0$ in $G$.
By Corollary~\ref{c:conj}, every principal $G_1$\nd subbundle of $P$ is isomorphic to $P_0tG_1$ for some $k$\nd point $t$ of $T_G(G_0,\, G_1)^{-1}$.
The isomorphism classes of principal $G_1$\nd subbundles of $P$ are thus parametrised by the set
\begin{equation}\label{e:Tset}
Z_G(G_0)(k) \backslash T_G(G_0,\, G_1)^{-1}(k)/G_1(k),
\end{equation} 
with the class of $t$ in $T_G(G_0,\, G_1)^{-1}(k)$ corresponding to the isomorphism class of $P_0tG_1$.
The set \eqref{e:Tset} depends only on $G$, $G_0$ and $G_1$.

Since $G_0$ is reductive, the $k$\nd scheme $\underline{\Hom}_k(G_0,\,G_1)$ of $k$\nd homomorphisms from $G_0$ to $G_1$ exists \cite[Proposition~1.3.3(i)]{O10},
and is the disjoint union of open affine subschemes which are homogeneous under the action by conjugation of $G_1$ \cite[Proposition~1.3.3(ii)]{O10}. 
The morphism of $k$\nd schemes from $T_G(G_0,\, G_1)^{-1}$ to $\underline{\Hom}_k(G_0,\,G_1)$
that sends $t$ to $g_0 \, \longmapsto \, t^{-1}g_0t$ is compatible with the actions of $Z_G(G_0)$ and $G_1$,
where $Z_G(G_0)$ acts trivially on $\underline{\Hom}_k(G_0,\,G_1)$.
It thus induces an embedding
\begin{equation*}
Z_G(G_0)(k) \backslash T_G(G_0,\, G_1)^{-1}(k)/G_1(k) \,\longrightarrow \, \Hom_k(G_0,\,G_1)/G_1(k)
\end{equation*}
of \eqref{e:Tset} into the set of $k$\nd homomorphisms up to conjugacy, with image the classes of those $G_0 \, \longrightarrow \, G_1$
whose composite with the embedding $G_1 \, \longrightarrow \, G$ is conjugate to the embedding $G_0 \, \longrightarrow \, G$.
It follows that $T_G(G_0,\, G_1)^{-1}$ is a finite disjoint union of open $k$\nd subschemes which are stable under 
$Z_G(G_0) \times_k G_1{}\!^\mathrm{op}$ and on whose $k$\nd points $(Z_G(G_0) \times_k G_1{}\!^\mathrm{op})(k)$ acts transitively. 
In particular, the set \eqref{e:Tset} is finite.

By generic flatness, $T_G(G_0,\, G_1)^{-1}$ is flat over $\underline{\Hom}_k(G_0,\,G_1)$, and hence smooth over it because 
the fibre above any $k$\nd point is isomorphic to $Z_G(G_0)$.
Thus $T_G(G_0,\, G_1)^{-1}$ is smooth over $k$ and hence reduced.
It follows that $T_G(G_0,\, G_1)^{-1}$ is a finite disjoint union of open homogeneous $(Z_G(G_0) \times_k G_1{}\!^\mathrm{op})$\nd subschemes.
Explicitly, if $t$ is a $k$\nd point of $T_G(G_0,\, G_1)^{-1}$, the composition morphism from $Z_G(G_0) \times_k tG_1$ to $T_G(G_0,\, G_1)^{-1}$
induces an isomorphism, compatible with the actions of $Z_G(G_0)$ and $G_1$, from
\begin{equation}\label{e:ZtG}
Z_G(G_0) \times_k^{Z_{tG_1t^{-1}}(G_0)} tG_1
\end{equation}
to the open homogeneous subscheme containing $t$.

Let $f\,:\,X' \,\longrightarrow\, X$ be a morphism of $k$\nd schemes with $H^0(X',\,\sO_{X'})$ also a henselian local $k$\nd algebra with residue field $k$.
For some reductive $k$\nd subgroup $G'{}\!_0$ of $G$ there exists a minimal principal $G'{}\!_0$\nd subbundle $P'{}\!_0$ of $f^*P_0$.
If isomorphism classes principal $G_1$\nd subbundles of $f^*P$ are parametrised similarly to the above with $G_0$ and $P_0$ replaced by $G'{}\!_0$ and $P'{}\!_0$,
then the map
\begin{equation}\label{e:Tmap}
Z_G(G_0)(k) \backslash T_G(G_0,\, G_1)^{-1}(k)/G_1(k) \, \longrightarrow \, Z_G(G'{}\!_0)(k) \backslash T_G(G'{}\!_0,\, G_1)^{-1}(k)/G_1(k)
\end{equation} 
defined by the embedding of $T_G(G_0,\, G_1)^{-1}$ into $T_G(G'{}\!_0,\, G_1)^{-1}$ corresponds to pullback of isomorphism classes along $f$.
In particular, if $f^*P_0$ is minimal, then $G'{}\!_0 \,=\, G_0$ and \eqref{e:Tmap} is bijective, so that for every reductive $k$\nd subgroup 
$G_1$ of $G$, pullback along $f$ defines a bijection from isomorphism classes of principal $G_1$\nd subbundles of $P$ 
to isomorphism classes of principal $G_1$\nd subbundles of $f^*P$.

Let $X$ be a non-empty connected proper $k$\nd scheme.
Then $H^0(X,\,\sO_X)$ is a finite local $k$\nd algebra.
Though it will not be needed in what follows, an explicit description can be given of the set of principal subbundles with reductive 
structure group of a given principal bundle $P$ over $X$, the fibration of this set over the set of isomorphism class of such subbundles, 
and the dependence of the fibres on $X$ and $P$.
This requires a semidirect product decomposition of gauge groups, which we first briefly recall.
If $G$ is an affine algebraic $k$\nd group and $P$ is a principal $G$\nd bundle, then for any affine $G$\nd scheme $Z$ the $k$\nd scheme
$\underline{H}^0(X,\,P \times_k^G Z)$ of cross-sections of $P \times_k^G Z$ over $X$, i.e.,\ the Weil restriction of $P \times_k^G Z$ from $X$ to $k$,
exists and is affine and of finite type \cite[Proposition~14.2]{O19}. 
In particular if $G'$ is an affine algebraic $k$\nd group on which $G$ acts by group automorphisms, then
\begin{equation*}
\underline{H}^0(X,\,P \times_k^G G')
\end{equation*}
is an affine algebraic $k$\nd group with Lie algebra $H^0(X,\,P \times_k^G \mathfrak{g}')$, where $\mathfrak{g}'$ is the Lie algebra of $G'$.
It contains as a $k$\nd subgroup the $k$\nd subgroup $G'{}^G$ of invariants of $G'$ under $G$.
Suppose that $G$ is reductive and that $P$ is minimal.
Then \cite[Theorem~14.5]{O19} we have a semidirect product decomposition
\begin{equation}\label{e:semi}
\underline{H}^0(X,\,P \times_k^G G')\, =\, U_G(P,\,\mathfrak{g}') \rtimes_k G'{}^G
\end{equation} 
in which \cite[Lemma~14.4]{O19} $U_G(P,\,\mathfrak{g}')$ is a normal unipotent $k$\nd subgroup of $\underline{H}^0(X,\,P \times_k^G G')$ 
characterised as follows: it is the unique connected $k$\nd subgroup with Lie algebra the ideal 
\begin{equation*}
{}^\mathrm{rad}H^0(X,\,P \times_k^G \mathfrak{g}')
\end{equation*}
of $H^0(X,\,P \times_k^G \mathfrak{g}')$, where ${}^\mathrm{rad}H^0(X,\,\sV)$ for any $\sO_X$\nd module $\sV$ denotes the kernel on the right 
of the pairing 
\begin{equation*}
\Hom_{\sO_X}(\sV,\sO_X) \otimes_k H^0(X,\,\sV)\, \longrightarrow \, H^0(X,\,\sO_X) \, \longrightarrow \, k,
\end{equation*}
with the first arrow defined by evaluation and the second by projection onto the residue field.
Further ${}^\mathrm{rad}H^0(X,\,\sV)$ is functorial in $\sV$ and hence $U_G(P,\,\mathfrak{g}')$ is functorial in $G'$.
If $G'$ is also reductive, then $G'{}^G$ is reductive and $U_G(P,\,\mathfrak{g}')$ is the unipotent radical of $\underline{H}^0(X,\,P \times_k^G G')$. 

Now let $G$ be an affine algebraic $k$\nd group, $P$ a principal $G$\nd bundle over $X$, and $G_1$ a reductive $k$\nd subgroup of $G$.
Then $G/G_1$ is an affine $G$\nd scheme of finite type, and the $k$\nd scheme
\begin{equation*}
S_{G_1}(P) \,=\, \underline{H}^0(X,\,P/G_1) \,=\, \underline{H}^0(X,\,P \times_k^G (G/G_1))
\end{equation*}
of principal $G_1$\nd subbundles of $P$ exists and is affine of finite type.
The $k$\nd group 
\begin{equation*}
A_G(P) \,=\, \underline{H}^0(X,\,\underline{\Aut}_G(P)) \,= \,\underline{H}^0(X,\,P \times_k^G G_\mathrm{conj})
\end{equation*}
of $G$\nd automorphisms of $P$ also exists and is affine of finite type, with the action of $A_G(P)$ on $S_{G_1}(P)$ that given by
the action of $G_\mathrm{conj}$ on $G/G_1$. 
Suppose that $S_{G_1}(P)$ is non-empty. 
Then $S_{G_1}(P)$ has a $k$\nd point.
Let $s$ be such a $k$\nd point, corresponding to the principal $G_1$\nd subbundle $P_1$ of $P$.
The stabiliser of $s$ under the action of $A_G(P)$ is the $k$\nd subgroup $A_{G_1}(P_1)$ of $A_G(P)$.
Thus we have an immersion 
\begin{equation}\label{e:imm}
A_G(P)/A_{G_1}(P_1) \, \longrightarrow \, S_{G_1}(P)
\end{equation}
of $A_G(P)$\nd schemes which sends the base $k$\nd point of $A_G(P)/A_{G_1}(P_1)$ to $s$.
If $\mathfrak{g}$ and $\mathfrak{g}_1$ are the Lie algebras of $G$ and $G_1$, then \eqref{e:imm}
induces on the tangent spaces at the base $k$\nd point and at $s$ the $k$\nd linear map
\begin{equation*}
H^0(X,\,P_1 \times_k^{G_1} \mathfrak{g})/H^0(X,\,P_1 \times_k^{G_1} \mathfrak{g}_1) \, \longrightarrow \,
H^0(X,\,P_1 \times_k^{G_1} \mathfrak{g}/\mathfrak{g}_1).
\end{equation*} 
Since $G_1$ is reductive, the projection from $\mathfrak{g}$ to $\mathfrak{g}/\mathfrak{g}_1$ is a retraction in the category
of representations of $G_1$.
It follows that \eqref{e:imm} induces an isomorphism at the tangent space at the base $k$\nd point and hence at every $k$\nd point.
This shows that the immersion \eqref{e:imm} is open.
Thus $S_{G_1}(P)$ is a finite disjoint union of open orbits under $A_G(P)$. 

Let $G_0$ be a reductive $k$\nd subgroup of $G$ for which $P$ has a minimal principal $G_0$\nd subbundle $P_0$.
Then the orbits of $S_{G_1}(P)$ under $A_G(P)$ are parametrised by the finite set \eqref{e:Tset}, with the orbit above 
the class of the $k$\nd point $t$ of $T_G(G_0,\, G_1)^{-1}$ that containing the $k$\nd point corresponding to $P_0tG_1$. 
We have a semidirect product decomposition
\begin{equation*}
A_G(P) \,=\, \underline{H}^0(X,\,P_0 \times_k^{G_0} G_\mathrm{conj}) \,=\, U_{G_0}(P_0,\,\mathfrak{g}) \rtimes_k Z_G(G_0)
\end{equation*}
of the form \eqref{e:semi}.
The stabiliser of the $k$\nd point of $S_{G_1}(P)$ corresponding to $P_0tG_1$ is the $k$\nd subgroup
\begin{equation*}
A_{G_1}(P_0tG_1) \,=\, \underline{H}^0(X,\,P_0 \times_k^{G_0} tG_1{}_\mathrm{conj}t^{-1})
\,=\, U_{G_0}(P_0,\,t\mathfrak{g}_1t^{-1}) \rtimes_k Z_{tG_1t^{-1}}(G_0),
\end{equation*}
of $A_G(P)$, where the embedding respects the semidirect product decompositions.
Further $U_{G_0}(P_0,\,t\mathfrak{g}_1t^{-1})$ is the unipotent
radical of $A_{G_1}(P_0tG_1)$ because $tG_1t^{-1}$ is reductive.
By the structure \eqref{e:ZtG} of the open orbits of $T_G(G_0,\, G_1)^{-1}$, the $Z_G(G_0)$\nd scheme $T_G(G_0,\, G_1)^{-1}/G_1$ is the disjoint
union of open homogeneous $Z_G(G_0)$\nd subschemes, parametrised by \eqref{e:Tset}, with the stabiliser of the $k$\nd point $\overline{t}$ 
corresponding to $tG_1$ the $k$\nd subgroup $Z_{tG_1t^{-1}}(G_0)$ of $Z_G(G_0)$. 
It follows that there exist unique morphisms 
\begin{equation*}
T_G(G_0,\, G_1)^{-1}/G_1 \, \xrightarrow{\,\,i\,\,} \, S_{G_1}(P) \, \xrightarrow{\,\,r\,\,} \, T_G(G_0,\, G_1)^{-1}/G_1
\end{equation*}
of $k$\nd schemes with $r \circ i$ is the identity, which respect the parametrisations by \eqref{e:Tset}, such that $i$ together with the embedding
of $Z_G(G_0)$ into $A_G(P)$ and $r$ together with the projection from $A_G(P)$ onto $Z_G(G_0)$ are compatible with the 
actions of $Z_G(G_0)$ and $A_G(P)$, and such that $i$ sends $\overline{t}$ to the $k$\nd point corresponding to $P_0tG_1$.
The fibre of $r$ above $\overline{t}$ is isomorphic as a $U_{G_0}(P_0,\,\mathfrak{g})$\nd scheme to
\begin{equation*}
U_{G_0}(P_0,\,\mathfrak{g})/U_{G_0}(P_0,\,t\mathfrak{g}_1t^{-1}),
\end{equation*}
and $U_{G_0}(P_0,\,t\mathfrak{g}_1t^{-1})$ is the fibre at $\overline{t}$ of a smooth group scheme over $T_G(G_0,\, G_1)^{-1}/G_1$
with restriction
\begin{equation*}
Z_G(G_0) \times_k^{Z_{tG_1t^{-1}}(G_0)} U_{G_0}(P_0,\,t\mathfrak{g}_1t^{-1})
\end{equation*} 
above the homogeneous component $Z_G(G_0)/Z_{tG_1t^{-1}}(G_0)$.

Proposition~\ref{p:spl} below is a particular case of
\cite[Corollary~10.14]{O19}
where $k$ is algebraically closed, $H \,=\, X$, and $G$ is of finite type.
It can be proved by an argument almost identical to that used for its analytic analogue Proposition~\ref{p:splan} below.

\begin{prop}\label{p:spl}
Let $X$ be a $k$\nd scheme, $G$ an affine algebraic $k$\nd group and $P$
a principal $G$\nd bundle over $X$.
Then $P$ has a principal $G_0$\nd subbundle for some reductive $k$\nd subgroup $G_0$ of $G$ if and only if 
the functor $P \times_k^G -$ on finite-dimensional representations of $G$ splits every short exact sequence.
\end{prop}

Let $X$ be a scheme with $\sO_{X,x}$ an integral domain for every $x \,\in\, X$. Then the 
irreducible components of $X$ are disjoint from one another. If $X$ is also locally noetherian, 
then the set of its irreducible components is locally finite, so that the irreducible 
components are open in $X$, and $X$ is their disjoint union.

\begin{thm}\label{t:redsub}
Let $f\,:\,X' \,\longrightarrow\, X$ be a morphism of $k$\nd schemes, $G$ an affine algebraic
$k$\nd group and $P$
a principal $G$\nd bundle over $X$. Suppose that either of the following two conditions holds:
\begin{enumerate}
\renewcommand{\theenumi}{(\alph{enumi})}
\item\label{i:redsublf}
$f_*\sO_{X'}$ is a locally free $\sO_X$\nd module of finite type which is nowhere $0$;
\item\label{i:redsubprop} 
$X$ is locally noetherian and normal, and $f$ is proper and surjective.
\end{enumerate}
Then $P$ has a principal $G_0$\nd subbundle for some reductive $k$\nd subgroup $G_0$ of $G$ if and only if $f^*P$ does. 
\end{thm}

\begin{proof}
The ``only if'' part is immediate, even without \ref{i:redsublf} or \ref{i:redsubprop}.
To prove the ``if'' part, suppose first that \ref{i:redsublf} holds.
Since the formation of $P \times^G -$ commutes with pullback, it is enough, by
Proposition~\ref{p:spl}, to show that a short exact sequence of vector bundles
\begin{equation*}
0 \,\longrightarrow\, \sV' \,\longrightarrow\, \sV \,\longrightarrow\,
\sV'' \,\longrightarrow\, 0,
\end{equation*}
associated to $P$ for a short exact sequence of representations of $G$,
splits if its pullback to $X'$ does.
This follows from Lemma~\ref{l:inj} with $\sW \,=\, \sV' \otimes_{\sO_X} \sV''^\vee$, because
the obstruction to splitting lies in $H^1(X,\,\sV' \otimes_{\sO_X} \sV''{}^\vee)$.

Suppose now that \ref{i:redsubprop} holds. Then $X$ is the disjoint union $\coprod_\alpha 
X_\alpha$ of its irreducible components $X_\alpha$. Let $X'{}\!_\alpha$ be a reduced and 
irreducible closed subscheme of $X'$ with generic point a point of $X'$ above the generic point 
of the component $X_\alpha$. To prove the ``if'' part, we may, after replacing $X'$ by $\coprod_\alpha 
X'{}\!_\alpha$, suppose that $f^{-1}(X_\alpha)$ is reduced and irreducible for each $\alpha$. By 
Lemma~\ref{l:restr} applied to the $X_\alpha$, \ref{i:redsublf} is then satisfied with $f$ 
replaced by its restriction $q\,:\,U'\,\longrightarrow\,
U$ above some open subscheme $U$ of $X$ with complement 
everywhere of codimension at least two. By Lemma~\ref{l:iso}, \ref{i:redsublf} is also satisfied 
with $f$ replaced by the inclusion map $j\,:\,U \,\hookrightarrow\, X$.
Since by hypothesis $f^*P$ and hence $(j \circ 
q)^*P$ has a principal subbundle with reductive structure group, the required result for 
\ref{i:redsubprop} follows by applying the result for \ref{i:redsublf} with $q$ and $j$
substituted for $f$.
\end{proof}

Let $f\,:\,X' \,\longrightarrow\, X$ be a morphism of $k$\nd schemes for which either \ref{i:redsublf} or \ref{i:redsubprop}
of Theorem~\ref{t:redsub} holds.
Suppose that $X$ is non-empty and connected.
Then $X$ has a non-empty connected open subscheme $U$ such that the restriction to $U$ of $f_*\sO_{X'}$ is locally free of finite type and 
the restriction above $U$ of any connected finite \'etale cover $X_1$ of $X$ is connected:
if \ref{i:redsublf} holds we may take $U = X$ and if \ref{i:redsubprop} holds we may by Lemma~\ref{l:restr} take for $U$ 
any sufficiently small non-empty open subscheme of $X$, because any $X_1$ is locally noetherian and normal and hence irreducible. 
If $X_1$ is a non-empty connected finite \'etale cover of $X$ whose pullback onto $X'$
has a cross-section,
then the restriction of $f$ above $U$ factors through the restriction $U_1$ of $X_1$ above $U$, so that if $U_1 = \Spec(\sR)$, then $f_*\sO_{X'}|U$
has a structure of $\sR$\nd module.
Thus $f_*\sO_{X'}|U$ is the push forward along $U_1 \, \xrightarrow{\,\,\,} \, U$ of a quasi-coherent $\sO_{U_1}$\nd module $\sV$, necessarily
locally free of constant rank.
The rank $r$ of $f_*\sO_{X'}|U$ is then the product of the degree of $X_1$ over $X$ and the rank of $\sV$.
Suppose that $X'$ is also non-empty and connected.
Then if $x'$ is a geometric point of $X'$ with $f(x') \,= \,x$, it follows that the continuous homomorphism
\begin{equation}\label{e:fhom}
\pi_1(X',\, x') \, \xrightarrow{\,\,\,\,} \, \pi_1(X,\,x)
\end{equation}
of profinite fundamental groups induced by $f$ has image in $\pi_1(X,\,x)$ of finite index dividing $r$.
It follows that $f$ factors essentially uniquely as a $k$\nd morphism $X' \,\longrightarrow\, X_1$ with $X_1$ connected
which induces a surjection on fundamental groups, followed by a finite \'etale $k$\nd morphism $X_1 \,\longrightarrow\, X$.
In the case where \ref{i:redsubprop} of Theorem~\ref{t:redsub} holds for $f$, it holds with $f$ replaced by 
$X' \,\longrightarrow\, X_1$, and, at least when $f$ is quasi-compact and quasi-separated, similarly in the case where \ref{i:redsublf}
of Theorem~\ref{t:redsub} holds for $f$.

\begin{lem}\label{l:rep}
Let $G$ be a reductive algebraic $k$\nd group, and let $G'\, \subset\, G$ be a reductive
$k$\nd subgroup.
Then the following conditions are equivalent:
\begin{enumerate}
\renewcommand{\theenumi}{(\alph{enumi})}
\item\label{i:repid} 
$G'$ contains the identity component of $G$;

\item\label{i:repfin}
there are only finitely many pairwise non-isomorphic irreducible representations $V$ of $G$ for
which $V^{G'} \,\ne\, 0$.
\end{enumerate}
\end{lem}

\begin{proof}
The group $G$ acts by left and right translation on its coordinate $k$\nd algebra $k[G]$.
Both \ref{i:repid} and \ref{i:repfin} are equivalent to the finiteness of the $k$\nd algebra $k[G]^{G'}$ of invariants 
under right translation: for \ref{i:repid} because $G/G'$ is affine with coordinate algebra $k[G]^{G'}$, and for \ref{i:repfin}
by the canonical decomposition of $k[G]$ as a $(G,\,G)$\nd bimodule.
\end{proof}

Let $G$ be a reductive algebraic $k$\nd group, and let $P$ be a principal $G$\nd bundle over a 
$k$\nd scheme $X$. We say that $P$ is \emph{almost minimal} if a principal $G'$\nd subbundle of 
$P$ exists for a reductive $k$\nd subgroup $G'$ of $G$ only if $G'$ contains the identity 
component of $G$.
Any push forward of an almost minimal principal $G$\nd bundle over $X$ along a surjective $k$\nd homomorphism is almost minimal. 

\begin{lem}\label{l:almin}
Let $X$ be a connected $k$\nd scheme, $G$ a reductive algebraic $k$\nd group,
and $G_0\, \subset\, G$ a $k$\nd subgroup containing 
the identity component of $G$. Let $P$ be a principal $G$\nd bundle over $X$ and $P_0$
a principal $G_0$\nd subbundle of $P$. Then $P$ is almost minimal if and only if $P_0$ is
so.
\end{lem}

\begin{proof}
The ``only if'' part is clear.
Conversely suppose that $P_0$ is almost minimal. 
Let $G_1$ be a reductive $k$\nd subgroup of $G$, and let $P_1$ be a
principal $G_1$\nd subbundle of $P$.
It is to be shown that $G_1$ contains the identity component $G^0$ of $G$.
For $i\,=\, 0,\,1$ denote by $\overline{G}_i$ the image of $G_i$
in the finite $k$\nd group $\overline{G} \,=\, G/G^0$. The push forward $\overline{P}_i$ of
$P_i$ along the quotient map $G_i\,\longrightarrow\,\overline{G}_i$
is a principal $\overline{G}_i$\nd subbundle of the push forward $\overline{P}$ of
$P$ along the quotient map $G \,\longrightarrow\,\overline{G}$. Furthermore, 
$P_0$ is the inverse image of $\overline{P}_0$ under the projection $P
\,\longrightarrow\, \overline{P}$.

Replacing $P_0$ by the inverse image in $P$ of a minimal principal subbundle of $\overline{P}_0$, we may suppose that $\overline{P}_0$ is minimal.
Replacing $G_0$ and $P_0$ by $gG_0g^{-1}$ and $P_0g^{-1}$ respectively for appropriate
$g \,\in\, G(k)$, we may further suppose, as in the two paragraphs following Theorem~\ref{t:push},
that $\overline{G}_1$ contains $\overline{G}_0$ and $\overline{P}_1$ contains $\overline{P}_0$.
The inverse image $P_0 \cap P_1$ of $\overline{P}_0$ under the smooth surjective morphism
$P_1\,\longrightarrow\, \overline{P}_1$ is then a principal
$(G_0 \cap G_1)$\nd subbundle of $P_0$.
Since $P_0$ is almost minimal, $G_0 \cap G_1$ and hence $G_1$ contains $G^0$, as required.
\end{proof}

If $X$ is a $k$\nd scheme for which $H^0(X,\,\sO_X)$ is a henselian local $k$\nd algebra, it 
follows from \cite[11.2(ii)]{O19}, with $\sC$ the tensor category $\Mod(X)$ of vector bundles 
over $X$, that $\Mod(X)$ has the Krull--Schmidt property, i.e.,\ that the commutative monoid 
under direct sum of isomorphism classes of objects of $\Mod(X)$ is free.

The case where \ref{i:redsublf} of Theorem~\ref{t:redsub} holds in Theorem~\ref{t:almin} below 
is equivalent to the case of \cite[15.5]{O19} where $k$ is algebraically closed, $H \,=\, X$,
and $K$ is of finite type; the proof for this case of Theorem~\ref{t:almin} is essentially the same.

\begin{thm}\label{t:almin}
Let $f\,:\,X'\,\longrightarrow\, X$ be a morphism of $k$\nd schemes, $G$
a reductive algebraic $k$\nd group and $P$ a principal $G$\nd bundle over $X$.
Suppose that $H^0(X,\,\sO_X)$ is henselian local $k$\nd algebra with residue field $k$.
Assume that either \ref{i:redsublf} or \ref{i:redsubprop} of Theorem~\ref{t:redsub} holds.
Then $P$ is almost minimal if and only if $f^*P$ is so.
\end{thm}

\begin{proof}
The ``if'' part is immediate, even without any conditions on $X$ or $f$.

To prove the converse suppose that $P$ is almost minimal.
Consider first the case where \ref{i:redsublf} of Theorem~\ref{t:redsub} holds.
Since $X$ is connected, $f_*\sO_{X'}$ is of constant rank $n$, so that $X'$
is a disjoint union of at most $n$ connected components.
Replacing $X'$ by such a component, we may assume that $X'$ is connected.
If $k^n$ is the standard $n$\nd dimensional representation of ${\rm GL}_n$, there is 
a principal ${\rm GL}_n$\nd bundle $P_0$ over $X$ such that $f_*\sO_{X'}$ is isomorphic to
$P_0 \times_k^{{\rm GL}_n} k^n$.
Let $G_1$ be a reductive $k$\nd subgroup of $G \times_k{\rm GL}_n$ such that
$P \times_X P_0$ has a minimal principal 
$G_1$\nd subbundle $P_1$.
Then $P$ and $P_0$ are the push forwards of $P_1$ along the projections
$G_1 \,\longrightarrow\, G$ and $G_1 \,\longrightarrow\, {\rm GL}_n$ respectively.
Thus the image of $G_1$ in $G$ contains the identity component of $G$ so that by
Lemma~\ref{l:almin} $f^*P$ is almost minimal if $f^*P_1$ is so.
Further if $k^n$ is regarded a representation of $G_1$, then $f_*\sO_{X'}$ is isomorphic to $P_1 \times_k^{G_1} k^n$.
Replacing $G$ and $P$ by $G_1$ and $P_1$ respectively, we may
therefore suppose that $P$ is minimal and that there is an isomorphism
\begin{equation*}
f_*\sO_{X'} \,\iso\, P \times_k^G W
\end{equation*}
for some finite-dimensional representation $W$ of $G$.

Let $G'$ be a reductive $k$\nd subgroup of $G$ for which $f^*P$ has a principal $G'$\nd subbundle $P'$.
Then for any representation $V$ of $G$, we have an isomorphism
\begin{equation*}
f^*(P \times_k^G V) \,=\, (f^*P) \times_k^G V \,\iso\, P' \times_k^{G'} V.
\end{equation*}
Suppose that $V$ is irreducible and $V^{G'} \,\ne\, 0$.
Then $V$ has, as a representation of $G'$, the trivial direct summand $k$, so that $f^*(P \times_k^G V)$ has the direct summand $\sO_{X'}$.
Applying $f_*$ shows that $f_*\sO_{X'}$ is a direct summand of
\begin{equation*}
f_*f^*(P \times_k^G V) \,=\, (P \times_k^G V) \otimes_{\sO_X} f_*\sO_{X'}
\,\iso\, (P \times_k^G V) \otimes_{\sO_X} (P \times_k^G W),
\end{equation*}
so that $P \times_k^G W$ is a direct summand of $P \times_k^G(V \otimes_k W)$.
By the Krull--Schmidt property of vector bundles over $X$ and Corollary~\ref{c:assoc}, $W$ is thus a direct summand of $V \otimes_k W$.
Since $V$ is irreducible and $W \ne 0$, it follows that $V$ is a direct summand of $W \otimes_k W^\vee$.
This shows that \ref{i:repfin} and hence \ref{i:repid} of Lemma~\ref{l:rep} holds.
Hence $f^*P$ is almost minimal. 

Consider now the case where \ref{i:redsubprop} of Theorem~\ref{t:redsub} holds.
After replacing $X'$ by the reduced subscheme of an irreducible component, we may suppose that $X'$ is reduced and irreducible.
Let $U$ be an open subscheme of $X$ as in Lemma~\ref{l:restr}.
Then
\begin{equation*}
H^0(U,\,\sO_U) \,=\, H^0(X,\,\sO_X)
\end{equation*}
by Lemma~\ref{l:iso}.
By Lemma~ \ref{l:iso} and the case where \ref{i:redsublf} of Theorem~\ref{t:redsub}
holds applied to the inclusion map $U\, \hookrightarrow\, X$ and to the restriction of $f$ above $U$,
we conclude that the restriction of $f^*P$ to $f^{-1}(U)$ is almost minimal.
Consequently, $f^*P$ is almost minimal.
\end{proof}

If $X$ and $X'$ are non-empty connected $k$\nd schemes and $f\,:\,X'\,\longrightarrow\, X$ is a morphism such that 
either \ref{i:redsublf} or \ref{i:redsubprop} of Theorem~\ref{t:redsub} holds,
it has been seen that the homomorphism \eqref{e:fhom} induced by $f$ on fundamental groups is ``almost surjective''.
The condition on pullbacks of connected finite \'etale covers in Theorem~\ref{t:pull} is equivalent
to the surjectivity of \eqref{e:fhom}.

\begin{thm}\label{t:pull}
Let $f\,:\,X'\,\longrightarrow\, X$ be a morphism of $k$\nd schemes.
Suppose that $H^0(X,\,\sO_X)$ and $H^0(X',\,\sO_{X'})$ are henselian local $k$\nd algebras
with residue field $k$. 
Suppose further that $f^*Z$ is connected for every connected finite \'etale cover $Z$ of $X$, and that either 
\ref{i:redsublf} or \ref{i:redsubprop} of Theorem~\ref{t:redsub} holds.
Then for every an affine algebraic $k$\nd group $G$, reductive $k$\nd subgroup $G_1$ of $G$, and principal $G$\nd bundle $P$ over $X$,
pullback along $f$ defines a bijection from the set isomorphism classes of principal $G_1$\nd subbundles of $P$
to the set isomorphism classes of principal $G_1$\nd subbundles of $f^*P$.
\end{thm}

\begin{proof}
If $P$ does not have a principal $G_0$\nd subbundle for any reductive $k$\nd subgroup $G_0$ of 
$G$, then by Theorem~\ref{t:redsub} neither does $f^*P$, so the sets of isomorphism classes in 
question are both empty. We may thus suppose that $P$ has a principal $G_0$\nd subbundle $P_0$ 
for some reductive $k$\nd subgroup $G_0$ of $G$. After replacing $G_0$ if necessary by a 
reductive $k$\nd subgroup, we may further suppose that $P_0$ is minimal. By the above, it then 
suffices to show that $f^*P_0$ is minimal.

Let $G'{}\!_0$ be a reductive $k$\nd subgroup of $G_0$ for which $f^*P_0$ has a principal 
$G'{}\!_0$\nd subbundle $P'{}\!_0$.
Then the scheme
\begin{equation*}
f^*(P_0/G'{}\!_0) \,\,= \,\,(f^*P_0)/G'{}\!_0
\end{equation*}
over $X'$ has a cross-section.
Since $f^*P_0$ is almost minimal by Theorem~\ref{t:almin},
we know that $G'{}\!_0$ contains the identity component of $G_0$.
Consequently, $P_0/G'{}\!_0$ is an \'etale cover of $X$.
Since $X$ is connected by the hypothesis on $H^0(X,\,\sO_X)$, and since pullback along
$f$ induces a bijection on connected components of \'etale covers,
it follows that $P_0/G'{}\!_0$ over $X$ has a cross-section, so that $P_0$ has a principal
$G'{}\!_0$\nd subbundle. Hence $G'{}\!_0 \,=\, G_0$ by minimality of $P_0$.
This shows that $f^*P_0$ is minimal.
\end{proof}

That some condition such as the one on $H^0(X,\,\sO_X)$ and $H^0(X',\,\sO_{X'})$ in 
Theorem~\ref{t:pull} is necessary can be seen from the following example.

Take for $X$ the 
affine line and for $f\,:\,X' \,\longrightarrow\, X$ a finite cover of $X$ such that
the genus of $X'$ is positive. Then a 
non-trivial line bundle over $X'$ exists, and hence a non-trivial decomposition of the trivial 
vector bundle of rank two over $X'$ as a direct sum of two line bundles, while no such 
decomposition exists over $X$. Therefore, if $G \,=\, {\rm GL}_2$, $G_1 \,=\,
\bG_m \times_k \bG_m$, and $P$ is the trivial
principal $G$\nd bundle over $X$, then $f^*P$ has a principal $G_1$\nd subbundle which 
is non-trivial and hence not isomorphic to the pullback of a principal $G_1$\nd subbundle of $P$.

\begin{cor}\label{c:pull}
Let $f\,:\,X'\,\longrightarrow\, X$ be a morphism of $k$\nd schemes such that the hypotheses of Theorem~\ref{t:pull} are satisfied.
\begin{enumerate}
\item\label{i:pullP}
For any reductive algebraic $k$\nd group $G$, two principal $G$\nd bundles $P_1$ and $P_2$ over $X$ are isomorphic
if and only if $f^*P_1$ and $f^*P_2$ are so.
\item\label{i:pullVindec}
A vector bundle $\sV$ over $X$ is indecomposable if and only if $f^*\sV$ is so.
\item\label{i:pullViso}
Two vector bundles $\sV_1$ and $\sV_2$ over $X$ are isomorphic if and only if $f^*\sV_1$ and $f^*\sV_2$ are so.
\end{enumerate}
\end{cor}

\begin{proof}
\ref{i:pullP}\,
$P_1$ and $P_2$ are isomorphic if and only if the principal $(G \times_k G)$\nd bundle $P_1 \times_k P_2$ over $X$ has a principal
$G$\nd subbundle, where $G$ is diagonally embedded in $G \times_k G$.
It thus suffices to apply Theorem~\ref{t:pull} with $G \times_k G$ for $G$ and $G$ for $G_1$.

\ref{i:pullVindec}\,
If $\sV$ has rank $n$, apply Theorem~\ref{t:pull} with $G = {\rm GL}_n$ and $G_1 = {\rm GL}_r \times_k {\rm GL}_s$ for $r+s = n$. 

\ref{i:pullViso}\,
If $\sV_1$ and $\sV_2$ have rank $n$, apply \ref{i:pullP} with $G = {\rm GL}_n$.
\end{proof}

The following is a slightly stronger condition on a morphism $f\,:\,X'
\,\longrightarrow\, X$ of $k$\nd schemes than \ref{i:redsubprop} of Theorem~\ref{t:redsub}:

\begin{itemize}
\item[$\mathrm{\ref{i:redsubprop}}^\prime$]
$X$ is locally noetherian and normal, $X'$ is non-empty, and $f$ is proper
and satisfying the condition that its restriction to every irreducible component
of $X'$ is surjective.
\end{itemize}

Suppose that either \ref{i:redsublf} of Theorem~\ref{t:redsub} or $\mathrm{\ref{i:redsubprop}}^\prime$ holds, that $X'$ is connected,
and that $H^0(X,\,\sO_X)$ is henselian local with residue field $k$.
Then it can be seen as follows that $H^0(X',\,\sO_{X'})$ is also henselian local with residue field $k$.

Note first that since $X'$ is non-empty and connected, $H^0(X',\,\sO_{X'})$ will have the required property provided that each of its elements 
is integral over $H^0(X,\,\sO_X)$.
If \ref{i:redsublf} of Theorem~\ref{t:redsub} holds, such an element $s$ is integral by the Cayley--Hamilton theorem applied to the endomorphism of 
$f_*\sO_{X'}$ defined by its global section $s$.

Suppose now that $\mathrm{\ref{i:redsubprop}}^\prime$ holds.
Consider first the case where the nilradical $\sN$ of $\sO_{X'}$ is nilpotent, i.e.,\ where 
\begin{equation*}
\sN^r \,=\, 0
\end{equation*}
for some $r$. Again we show that in this case every element $s$ of $H^0(X',\,\sO_{X'})$ is 
integral over $H^0(X,\,\sO_X)$. We may assume that $X'$ is reduced, because if $\overline{s}$ 
is the image of $s$ in $H^0(X'{}_\mathrm{red},\,\sO_{X'{}_\mathrm{red}})$, then $s$ is 
annihilated by the $r$th power of any monic polynomial over $H^0(X,\,\sO_X)$ which annihilates 
$\overline{s}$. If $U$ is a non-empty affine open subscheme of $X$, then $f^{-1}(U)$ is a 
noetherian open subscheme of $X'$ which has $Z \cap f^{-1}(U)$ as an irreducible component for 
every irreducible component $Z$ of $X'$. Thus $X'$ has only finitely many irreducible 
components $Z$, and since $H^0(X',\,\sO_{X'})$ embeds into the product of the $H^0(Z,\,\sO_Z)$ 
over such $Z$, we may further assume that $X'$ is irreducible. Then by Lemma~\ref{l:restr}, 
\ref{i:redsublf} holds with $f$ replaced by its restriction $U'\,\longrightarrow\, U$ above 
some open subscheme $U$ of $X$ with complement everywhere of codimension at least two. Consequently, from 
Lemma~\ref{l:iso} and the case where \ref{i:redsublf} holds, it
is deduced that the restriction of $s$ to $U'$, 
and hence also $s$ itself, is annihilated by a monic polynomial with coefficients in 
$H^0(X,\,\sO_X)$.

Consider now the general case where $\mathrm{\ref{i:redsubprop}}^\prime$ holds. Applying what 
has just been seen with $X'$ replaced by its closed subscheme with structure sheaf 
$\sO_{X'}/\sN^r$ shows that $H^0(X',\,\sO_{X'}/\sN^r)$ is henselian local with residue field $k$ 
for each $r$. It is thus enough to show that the canonical homomorphism
\begin{equation*}
H^0(X',\,\sO_{X'})\,\,\longrightarrow\,\, \lim_r H^0(X',\,\sO_{X'}/\sN^r) 
\end{equation*}
is an isomorphism.
Since the restriction of $\sN$ to every affine open subset of $X'$ is nilpotent, this can be done by writing the spaces of global sections 
as limits of spaces of sections over the open subsets of an affine open cover of $X'$ and their intersections, 
and then interchanging these limits with $\lim_r$.

It follows from the above that if we replace in Theorem~\ref{t:pull} the hypothesis that 
``either \ref{i:redsublf} or \ref{i:redsubprop} of Theorem~\ref{t:redsub} holds''
by the slightly stronger one that
``either \ref{i:redsublf} of Theorem~\ref{t:redsub} or $\mathrm{\ref{i:redsubprop}}^\prime$ holds'', 
then the hypothesis ``$H^0(X,\,\sO_X)$ and $H^0(X',\,\sO_{X'})$ are henselian local $k$\nd algebras
with residue field $k$'' may be replaced by ``$H^0(X,\,\sO_X)$ is a henselian local $k$\nd algebra
with residue field $k$''.
Indeed the connectedness condition on $X'$ is satisfied because $X'$ is the pullback of the identity \'etale cover of $X$ along $f$.

Let $X$ be a $k$\nd scheme and $\sR$ a finite locally free $\sO_X$\nd algebra which is nowhere $0$.
Then $H^0(X,\,\sO_X)$ is a local $k$\nd algebra if $H^0(X,\,\sR)$ is:
if $\iota\,:\,\sO_X\,\longrightarrow\, \sR$ is the unit and $\tr\,:\,\sR \,\longrightarrow\, \sO_X$ is the trace,
then for $a\, \in\, H^0(X,\,\sO_X)$ and $a'\,\in\, H^0(X,\,\sR)$,
\begin{equation*}
\tr(\iota(a)a') \,=\, a\tr(a'),
\end{equation*}
so that $\iota$ preserves non-units.
Reducing to the case where $X$ is connected and $\sR$ is indecomposable then shows that 
$H^0(X,\,\sO_X)$ is a finite product local $k$\nd algebras if
$H^0(X,\,\sR)$ is so.
If $R_1$ is a finite free $H^0(X,\,\sO_X)$\nd algebra and $X_1$ is the fibre product of $X$ with 
$\Spec(R_1)$ over $\Spec(H^0(X,\,\sO_X))$, then
\begin{equation*}
H^0(X_1,\,\sO_{X_1}) = R_1,
\end{equation*}
the pullback $\sR_1$ of $\sR$ onto $X_1$ is finite locally free and nowhere $0$, and
\begin{equation*}
H^0(X_1,\,\sR_1) = H^0(X,\,\sR) \otimes_{H^0(X,\,\sO_X)} R_1.
\end{equation*}
Since $H^0(X_1,\,\sO_{X_1})$ is a finite product of local $k$\nd algebras if $H^0(X_1,\,\sR_1)$ is so, it follows that
$H^0(X,\,\sO_X)$ is a henselian local $k$\nd algebra if $H^0(X,\,\sR)$ is so.

If $f\,:\,X' \,\longrightarrow\, X$ is a morphism of $k$\nd schemes for which \ref{i:redsublf} of Theorem~\ref{t:redsub} holds, the above
with $\sR\,=\, f_*\sO_{X'}$ shows that if $H^0(X',\,\sO_{X'})$ is a henselian local $k$\nd algebra, then so is $H^0(X,\,\sO_X)$.

If $X$ is a $k$\nd scheme which is locally of finite type, an argument similar to the proof of 
\cite[Lemma~2.1]{BisO'S21} shows that $H^0(X,\,\sO_X)$ is henselian local with residue field 
$k$ if and only if $X$ is non-empty and the restriction to $X_\mathrm{red}$ of any regular 
function on $X$ is constant. Therefore, if $X' \,\longrightarrow\, X$ is a dominant morphism of $k$\nd schemes which 
are locally of finite type, and if $H^0(X',\,\sO_{X'})$ is henselian local with residue field 
$k$, then so is $H^0(X,\,\sO_X)$.

Consider finally the following condition stronger than \ref{i:redsubprop} of 
Theorem~\ref{t:redsub}:

\begin{itemize}
\item[$\mathrm{\ref{i:redsubprop}}^{\prime\prime}$]
$X$ is locally of finite type and normal, and $f$ is proper and surjective.
\end{itemize}

It follows from the above that if in Theorem~\ref{t:pull} the hypothesis that 
``either \ref{i:redsublf} or \ref{i:redsubprop} of Theorem~\ref{t:redsub} holds'' is replaced by
``either \ref{i:redsublf} of Theorem~\ref{t:redsub} or
$\mathrm{\ref{i:redsubprop}}^{\prime\prime}$ holds'',
then the hypothesis ``$H^0(X,\,\sO_X)$ and $H^0(X',\,\sO_{X'})$ are henselian local $k$\nd algebras
with residue field $k$'' may be replaced by ``$H^0(X',\,\sO_{X'})$ is a henselian local $k$\nd algebra
with residue field $k$''.

\section{Complex analytic spaces}

Let $X$ be a complex analytic space and $J$ a complex Lie group. By a principal $J$\nd 
bundle over $X$ is meant a complex analytic space $P$ over $X$ together with a right action of 
$J$ on $P$ above $X$, such that $P$ is locally over $X$ isomorphic to $X \times J$ with $J$ 
acting by right translation. The pullback $f^*P$ of $P$ along a morphism of complex analytic 
spaces $f\,:\,X' \,\longrightarrow\, X$, with its canonical right action of $J$ over $X$, is
a principal $J$\nd bundle over $X'$.

Given a homomorphism $h\,:\,J \,\longrightarrow\, J'$ of complex Lie groups and a principal $J$\nd bundle $P$ 
over $X$, there exists a principal $J'$\nd bundle $P'$ over $X$ and a morphism $l\,:\,P \,\longrightarrow\, P'$ 
over $X$ such that $l(pj) \,=\, l(p)h(j)$ for points $p$ of $P$ and $j$ of $J$, and the pair $(P',\,l)$ 
is unique up to a unique isomorphism. The principal $J'$\nd bundle $P'$ is the push forward of 
$P$ along $h$.

Let $J_1$ be a closed complex Lie subgroup of $J$. A \emph{principal $J_1$\nd subbundle of $P$} 
is defined as a closed analytic subspace $P_1$ of $P$ such that the action of $J$ on $P$ 
restricts to an action of $J_1$ on $P_1$ with $P_1$ a principal $J_1$\nd bundle over $X$. Two 
principal $J_1$\nd subbundles of $P$ will be called isomorphic if they are isomorphic as 
principal $J_1$\nd bundles over $X$. By uniqueness of push forward, any such isomorphism is the 
restriction of a unique automorphism of the principal $J$\nd bundle $P$ over $X$.

By a representation of a complex Lie group $J$ we mean a complex vector space $V$ together with an $\sO_J$\nd automorphism of the pullback
of $V$ onto $J$ satisfying the usual associativity property for an action.
If $G$ is an affine algebraic $\C$\nd group, any representation of $G$ may be regarded as a representation of its associated
complex Lie group $G_\mathrm{an}$ by pulling back the action of $G$ along the canonical morphism $G_\mathrm{an} \,\longrightarrow\, G$.

Let $P$ be a principal $J$\nd bundle over $X$.
We have an action of $J$ on $P$ with $jp = pj^{-1}$ for points $j$ of $J$ and $p$ of $P$. 
If $V$ is a representation of $J$, arguing locally over $X$ shows that a pair consisting of an $\sO_X$\nd module $\sV$ and
a $J$\nd equivariant isomorphism from the pullback of $V$ to $P$ to the pullback of $\sV$ to $P$ exists, and is
unique up to a unique isomorphism. We write $\sV$ as
\begin{equation*}
P \times^J V.
\end{equation*} 
It is the usual $\sO_X$\nd module associated to $V$ by identifying the points $(pj,\,v)$ and $(p,\,jv)$ of $P \times V$.
If $V$ is finite-dimensional of dimension $n$, then $P \times^J V$ is a vector bundle everywhere of rank $n$.
Formation of $P \times^J V$ is functorial in $P$, $J$ and $V$ and is compatible with pullback.
Similarly we define a complex analytic space 
\begin{equation*}
P \times^J Z
\end{equation*}
over $X$ for $Z$ a complex analytic space with an action of $J$.

A commutative $\sO_X$\nd algebra will be said to be \emph{of finite presentation} if locally on $X$
it is isomorphic to a quotient by a finite number of sections of a polynomial $\sO_X$\nd algebra
in finite number of indeterminates.
Commutative $\sO_X$\nd algebras of finite presentation are closed under the formation 
of finite colimits of commutative $\sO_X$\nd algebras. If $a\,:\,Z\,\longrightarrow\, X$
is a morphism of complex analytic spaces and the commutative $\sO_X$\nd algebra $\sR$
is of finite presentation, then the commutative $\sO_Z$\nd algebra $a^*\sR$ is of finite presentation.

Let $\sR$ be a commutative $\sO_X$\nd algebra of finite presentation. Then the contravariant functor on
complex analytic spaces over $X$ that sends $Z$ with structural morphism $a\,:\,Z \,\longrightarrow\, X$ to 
\begin{equation*}
\Hom_{\sO_Z\text{-}\mathrm{alg}}(a^*\sR,\,\,\sO_Z)
\end{equation*}
is representable, and we write the representing object as
\begin{equation*}
\Sp(\sR).
\end{equation*}
To see that $\Sp(\sR)$ exists, note that it exists if it does so locally on $X$,
and that if $\sR$ is a polynomial algebra in $n$ variables then $\Sp(\sR)$ is an affine $n$\nd space over $X$.
Thus $\Sp(\sR)$ exists for any $\sR$ of finite presentation, with $\Sp$ sending finite colimits to finite limits, 
because $\sR$ is locally on $X$ a coequaliser of two morphisms between polynomial algebras.
Formation of $\Sp(\sR)$ commutes with pullback.
If $X$ is an infinitesimal thickening of a point, and hence is a $\C$\nd scheme, then
$\Sp(\sR)$ is the complex analytic space over $X$ associated to $\Spec(\sR)$.

The assignment $\sR \, \longmapsto \, \Sp(\sR)$ extends canonically to a contravariant functor from 
commutative $\sO_X$\nd algebras of finite presentation to complex analytic spaces over $X$, which
restricting to the case where $X$ is infinitesimal shows to be faithful.

Recall that if $G$ and $G_0$ are affine algebraic $\C$\nd groups with $G_0$ reductive, then any homomorphism
$G_0{}_\mathrm{an}\,\longrightarrow\, G_\mathrm{an}$ of complex Lie groups is of the form $h_\mathrm{an}$
for a (unique) $\C$\nd homomorphism $h\,:\,G_0\,\longrightarrow\, G$. In particular the functor $G\,\longmapsto\,
 G_\mathrm{an}$ from reductive algebraic $\C$\nd groups to complex Lie groups is fully faithful. A complex
Lie group will be called \emph{algebraic} if it is isomorphic to the complex Lie group $G_\mathrm{an}$ associated to 
some affine algebraic $\C$\nd group $G$, and \emph{reductive} if further $G$ may be taken to be reductive.

Let $J$ be an algebraic complex Lie group. If we write $J \,=\,
G_\mathrm{an}$ for an affine algebraic $\C$\nd group $G$ with $G \,=\, \Spec(R)$, then
\begin{equation*}
J \,=\, \Sp(R).
\end{equation*}
Further if $R$ is equipped with the action of $G$ defined by conjugation, the corresponding action of $J$
on $R$ induces conjugation on $J$.
The group space over $X$ of automorphisms of a principal $J$\nd bundle $P$ over $X$ is then
\begin{equation*}
\underline{\Aut}_J(P)\, =\, P \times^J J_\mathrm{conj}\, =\, \Sp(\sR)
\end{equation*}
with $\sR$ the $\sO_X$\nd algebra $P \times^J R_\mathrm{conj}$, where the subscripts ${}_\mathrm{conj}$ indicate 
that the action of $J$ is defined by conjugation.
If $J_1$ is a reductive closed complex Lie subgroup of $J$, there is a (unique) reductive $\C$\nd subgroup $G_1$ of $G$
with $G_1{}_\mathrm{an} \,=\, J_1$. Then 
\begin{equation*}
J/J_1\, =\, (G/G_1)_\mathrm{an}\, =\, \Sp(R_1),
\end{equation*}
where $G/G_1 \,=\, \Spec(R_1)$, and
\begin{equation*}
P/J_1 \,=\, P \times^J J/J_1\, =\, \Sp(\sR_1),
\end{equation*}
where $\sR_1\, =\, P \times^J R_1$ with the action of $J$ on $R_1$ inducing the action by left translation of $J$ on $J/J_1$.

As in the algebraic case, principal $J_1$\nd subbundles of a principal $J$\nd bundle $P$ over $X$ correspond to cross-sections of $P/J_1$
over $X$. Proposition~\ref{p:bijan} below can now be proved in the same way as Proposition~\ref{p:bij}.

\begin{prop}\label{p:bijan}
Let $f\,:\,X' \,\longrightarrow\, X$ be a morphism of complex analytic spaces with $\sO_X \,\longrightarrow\,
f_*\sO_{X'}$ an isomorphism. Let $J$ be an algebraic complex Lie group and $J_1\,\subset\, J$ a reductive closed complex Lie subgroup.
Then for any principal $J$\nd bundle $P$ over $X$, pullback along the projection of $f^*P$ onto $P$ defines a bijection
from the set of (respectively, the set of isomorphism classes of) principal $J_1$\nd subbundles of $P$ to the set of 
(respectively, the set of isomorphism classes of) principal $J_1$\nd subbundles of $f^*P$.
\end{prop}

By a tensor category we mean a $\C$\nd linear category with a $\C$\nd bilinear tensor product, together with a unit $\I$ and associativity and 
commutativity constraints satisfying the usual compatibilities \cite[p. 105, Definition~1.1]{DM82}.
A tensor functor between tensor categories is a $\C$\nd linear functor together with structural isomorphisms compatible with the constraints
ensuring that the unit and tensor product are preserved up to isomorphism \cite[pp. 113--114, Definition~1.8]{DM82}.
A tensor isomorphism between tensor functors is a natural isomorphism which is compatible with the structural
isomorphisms \cite[p. 116, Definition~1.12]{DM82}.

Let $X$ be a complex analytic space.
The category of $\sO_X$\nd modules, with the usual unit, tensor product and constraints, is a tensor category
\begin{equation*}
\MOD(X).
\end{equation*}
The vector bundles over $X$, identified with the locally free $\sO_X$\nd modules of finite type, form a full tensor subcategory $\Mod(X)$.

Let $G$ be an affine algebraic $\C$\nd group.
We denote by 
\begin{equation*}
\REP_{\C}(G)
\end{equation*}
the tensor category of representations of $G$, and by $\Rep_{\C}(G)$ the full tensor subcategory of finite-dimensional representations.
Let
\begin{equation*}
T\,:\,\Rep_{\C}(G) \,\longrightarrow\, \Mod(X)
\end{equation*}
be a tensor functor which is exact, in the sense that it is exact as a functor to $\MOD(X)$.
Writing representations of $G$ as the filtered colimit of their finite-dimensional subrepresentations 
shows that $T$ extends to a tensor functor
\begin{equation*}
\widehat{T}\,:\,\REP_{\C}(G) \,\longrightarrow\, \MOD(X)
\end{equation*}
from the category of all representations of $G$, which is exact and preserves colimits.
Similarly every tensor isomorphism $\theta\,:\,T' \,\iso\, T$ extends uniquely to a tensor isomorphism 
$\widehat{\theta}\,:\,\widehat{T'} \,\iso\, \widehat{T}$.
If $Z$ is a complex analytic space over $X$, we write
\begin{equation*}
T_Z\,:\,\Rep_{\C}(G) \,\longrightarrow\, \Mod(Z)
\end{equation*}
for $T$ followed by pullback along $Z \,\longrightarrow\, X$, and similarly for $\widehat{T}$.

When $X$ is a point, $\Mod(X)$ is the category $\Mod(\C)$ of finite-dimensional $\C$\nd vector spaces, and $\MOD(X)$ the category $\MOD(\C)$ of all 
$\C$\nd vector spaces.
If we take for $T$ the forgetful tensor functor
\begin{equation*}
\omega\,:\,\Rep_{\C}(G) \,\longrightarrow\, \Mod(\C),
\end{equation*}
then $\widehat{\omega}\,:\,\REP_{\C}(G) \,\longrightarrow\, \MOD(\C)$ is also the forgetful tensor functor, and
\begin{equation*}
\widehat{\omega}_Z\,\,:\,\,\REP_{\C}(G)\,\,\longrightarrow\,\, \MOD(Z)
\end{equation*}
sends a representation to the free $\sO_Z$\nd module on its underlying vector space.

The tensor functor that sends a representation of $G$ to its underlying vector space with the trivial action of $G$
will be denoted by
\begin{equation*}
E\,:\,\REP(G) \,\longrightarrow\, \REP(G).
\end{equation*} 
Then $\widehat{T}E \,=\, \widehat{\omega}_X$ for any exact tensor functor $T\,:\,\Rep_{\C}(G)\,\longrightarrow\, \Mod(X)$.

We regard the coordinate algebra $\C[G]$ as the left regular representation of $G$, 
where the point $g$ of $G$ sends $w$ in $\C[G]$ to $w(g^{-1}-)$.
Then $\C[G]$ is a commutative $G$\nd algebra under pointwise multiplication.
The action of $G$ on a representation $V$ of $G$
is an isomorphism
\begin{equation}\label{e:Gactiso}
E(V) \otimes_{\C} \C[G] \,\,\iso\,\, V \otimes_{\C} \C[G]
\end{equation}
of modules over the commutative algebra $\C[G]$ in $\REP(G)$.

If $T$ and $T'$ are tensor functors from $\Rep_{\C}(G)$ to $\Mod(X)$, we write $\underline{\Iso}_X^\otimes(T',\,T)$ 
for the functor on complex analytic spaces over $X$ with
\begin{equation*}
\underline{\Iso}_X^\otimes(T',\,T)(Z) \,=\, \Iso^{\otimes}(T'{}\!_Z,\,T_Z),
\end{equation*}
where $\Iso^{\otimes}$ denotes the set of tensor isomorphisms.
The pullback of $\underline{\Iso}_X^\otimes(T',\,T)$ along $X' \,\longrightarrow\,
X$ is $\underline{\Iso}_{X'}^\otimes(T'{}\!_{X'},\,T_{X'})$. When $X$ is a
point, we also write $\underline{\Iso}^\otimes(T',\,T)$ for $\underline{\Iso}_X^\otimes(T',\,T)$.

The action of $G_{\mathrm{an}}$ for any $V$ in $\Rep_{\C}(G)$ is an automorphism of 
the $\sO_{G_{\mathrm{an}}}$\nd module $\omega_{G_{\mathrm{an}}}(V)$.
Such automorphisms are the components of an action tensor automorphism of $\omega_{G_{\mathrm{an}}}$.
The natural transformation $\alpha$ from the functor represented by 
$G_{\mathrm{an}}$ to $\underline{\Iso}^\otimes(\omega,\,\omega)$ defined by the action tensor automorphism 
in $\underline{\Iso}^\otimes(\omega,\,\omega)(G_{\mathrm{an}})$ has component
\begin{equation*}
\alpha_Z:G_{\mathrm{an}}(Z)\,\longrightarrow\, \Iso^{\otimes}(\omega_Z,\omega_Z)
\end{equation*} 
at $Z$ the map with $\alpha_Z(g)_V$ the action of $g$ on $V$.
Further $\widehat{\alpha_Z(g)}_V$ is the action of $g$ on $V$ for any $V$ in $\REP_{\C}(G)$.
By associativity of the action of $G_{\mathrm{an}}$, each $\alpha_Z$ is a group homomorphism.

\begin{lem}\label{l:Grep}
For any affine algebraic $\C$\nd group $G$, the functor $\underline{\Iso}^\otimes(\omega,\,\omega)$ is represented by 
$G_{\mathrm{an}}$ with universal element in $\underline{\Iso}^\otimes(\omega,\,\omega)(G_{\mathrm{an}})$ the action tensor automorphism.
\end{lem}

\begin{proof}
Let $Z$ be a complex analytic space.
Then
\begin{equation*}
\Sp(\widehat{\omega}_Z(\C[G])) \,=\, \Sp(\widehat{\omega}_Z(E(\C[G]))) \,=\, G_{\mathrm{an}}{}_Z,
\end{equation*}
and with $\alpha$ as above the automorphism $\Sp(\widehat{\alpha_Z(g)}_{\C[G]})$ of $G_{\mathrm{an}}{}_Z$ is left translation by $g^{-1}$.

Let $\theta$ be a tensor automorphism of $\omega_Z$.
It is to be shown that $\theta = \alpha_Z(g)$ for a unique $g \in G_{\mathrm{an}}(Z)$.
Since $\widehat{\theta}_{E(\C[G])}$ is the identity,
applying $\Sp \widehat{\theta}$ to the morphism
\begin{equation*}
\C[G] \,\longrightarrow\, \C[G] \otimes_{\C} E(\C[G]) 
\end{equation*}
of commutative algebras in $\REP(G)$ induced by the composition $G \times_{\C} G\,\longrightarrow\,G$ 
shows that the automorphism $\Sp(\widehat{\theta}_{\C[G]})$
of the complex analytic space $G_{\mathrm{an}}{}_Z$ over $Z$ is compatible with right translation.
It is therefore the left translation by $g^{-1}$ for some cross-section $g$ of $G_{\mathrm{an}}{}_Z$, so that
\begin{equation*}
\widehat{\theta}_{\C[G]} \,=\, \widehat{\alpha_Z(g)}_{\C[G]}.
\end{equation*}
For any $V$ in $\Rep_{\C}(G)$, composing the inverse of \eqref{e:Gactiso} with the embedding of $V$ into $V \otimes_{\C} \C[G]$
defined by $1 \in \C[G]$ gives a monomorphism 
\begin{equation*}
V \,\longrightarrow\, E(V) \otimes_{\C} \C[G]
\end{equation*}
in $\REP_{\C}(G)$ from $V$ to a direct sum of copies of $\C[G]$.
Consequently, $\theta_V \,=\, \alpha_Z(g)_V$ for every $V$, so that $\theta \,=\, \alpha_Z(g)$.
Setting $V$ to be faithful then shows that $g$ is unique.
\end{proof}

Let $\sR$ be a commutative $\sO_X$\nd algebra of finite presentation.
If $b\,:\,\Sp(\sR) \,\longrightarrow\, X$ is the structural morphism, the universal morphism makes
$\sO_{\Sp(\sR)}$ a $b^*\sR$\nd algebra.
We then have a tensor functor $\sM \,\longmapsto \,\widetilde{\sM}$ with
\begin{equation*}
\widetilde{\sM} \,=\, b^*\sM \otimes_{b^*\sR} \sO_{\Sp(\sR)}
\end{equation*}
from $\sR$\nd modules to $\sO_{\Sp(\sR)}$\nd modules. Further
\begin{equation*}
(\sF \otimes_{\sO_X} \sR)^{\sim} \,=\, b^*\sF
\end{equation*}
for an $\sO_X$\nd module $\sF$.

Suppose that $\sR$ is flat as an $\sO_X$\nd module and that the spectra of its fibres are smooth $\C$\nd schemes.
Then it can be seen as follows that $\mathrm{Sp}(\sR)$ is smooth over $X$.
By restricting above infinitesimal neighbourhoods of points in $X$, we reduce to the case where
$X$ is an infinitesimal thickening of a point.
The scheme $\Spec(\sR)$ over $X$ is then smooth because it is flat with smooth fibres,
so that $\mathrm{Sp}(\sR)$ is smooth.

By Lemma~\ref{l:Grep}, $G_{\mathrm{an}}$ may be identified with $\underline{\Iso}^\otimes(\omega,\,\omega)$ and hence 
$G_{\mathrm{an}}{}_X$ with $\underline{\Iso}_X^\otimes(\omega_X,\,\omega_X)$.
For any exact tensor functor $T$ from $\Rep_{\C}(G)$ to $\Mod(X)$, composition of tensor isomorphisms 
thus defines a canonical right action of $G_{\mathrm{an}}$ on $\underline{\Iso}_X^\otimes(\omega_X,\,T)$. 

\begin{lem}\label{l:principal}
Let $X$ be a complex analytic space, $G$ an affine algebraic $\C$\nd group and $T$ an exact tensor functor from $\Rep_{\C}(G)$
to $\Mod(X)$.
Then the functor $\underline{\Iso}_X^\otimes(\omega_X,\,T)$ is representable, and is a principal $G_{\mathrm{an}}$\nd bundle over $X$
for the canonical right action of $G_{\mathrm{an}}$.
\end{lem}

\begin{proof}
The commutative algebra $\C[G]$ in $\REP_{\C}(G)$ can be written as the coequaliser of two morphisms
between symmetric algebras on finite-dimensional representations of $G$. 
Thus the $\sO_X$\nd algebra $\widehat{T}(\C[G])$ is of finite presentation.
Consequently, we have a complex analytic space
\begin{equation*}
P \,=\, \mathrm{Sp}(\widehat{T}(\C[G]))
\end{equation*}
over $X$.
Applying $\widehat{T}$ and then $(-)^\sim$ to \eqref{e:Gactiso} gives an isomorphism
\begin{equation*}
\omega_P(V)\, =\, (\omega_X(V) \otimes_{\sO_X} \widehat{T}(\C[G]))^\sim 
\,\iso\, (T(V) \otimes_{\sO_X} \widehat{T}(\C[G]))^\sim \,=\, T_P(V)
\end{equation*}
for $V$ in $\Rep_{\C}(G)$. 
It is the component at $V$ of a tensor isomorphism
\begin{equation*}
u\,:\,\omega_P \,\iso\, T_P.
\end{equation*}
Formation of $(P,\,u)$ is compatible with pullback along morphisms $Z \,\longrightarrow\, X$, and with tensor isomorphisms $T
\,\iso\, T'$.
We show that $P$ represents the functor $\underline{\Iso}_X^\otimes(\omega_X,\,T)$ with universal element $u$, 
and that $P$ is a principal $G_{\mathrm{an}}$\nd bundle over $X$
for the right action of $G_{\mathrm{an}}$ defined by the canonical one on $\underline{\Iso}_X^\otimes(\omega_X,\,T)$.

Suppose first that $X$ is a point and $T \,=\, \omega$. Then $P \,=\, G_\mathrm{an}$, and $u$ is the action tensor automorphism.
The required results then follow from Lemma~\ref{l:Grep}, because the canonical right action of $G_\mathrm{an}$ on 
$\underline{\Iso}^\otimes(\omega,\,\omega)$ is by right translation.

To prove the required results for arbitrary $X$ and $T$, it is enough by compatibility of $(P,\,u)$ with pullback and tensor isomorphisms 
to show that each point of $X$ is contained in an open subset $U$ such that $T_U$ is tensor isomorphic to $\omega_U$. 

The fibre of $\widehat{T}(\C[G])$ at any $x$ in $X$ is non-zero because the fibre 
$\widehat{T}_x$ of $\widehat{T}$ at $x$ is exact and $\C \,\longrightarrow\, \C[G]$ is a monomorphism.
Consequently, $P_x$ is non-empty. Any point of $P_x$ defines a tensor isomorphism from $\omega$ to $T_x$, and 
hence from $\widehat{\omega}$ to $\widehat{T}_x$. Thus $\widehat{\omega}(\C[G])$ and 
$\widehat{T}_x(\C[G])$ are isomorphic $\C$\nd algebras, so that the spectrum of 
$\widehat{T}_x(\C[G])$ is smooth. Also $\widehat{T}(\C[G])$ is flat over $\sO_X$, because it is 
the filtered colimit of the vector bundles $\widehat{T}(V)$ over $X$ as $V$ runs over the 
finite dimensional subrepresentations of $\C[G]$. Therefore, $P \,\longrightarrow\, X$ is smooth and surjective. It 
follows that any point of $X$ is contained in an open subspace of $X$ over which $P$ has a 
section. Any such section defines the required tensor isomorphism.
\end{proof}

Let $P$ be a principal $G_\mathrm{an}$\nd bundle over $X$.
Then we have an exact tensor functor
\begin{equation*}
P \times^{G_\mathrm{an}} -\,:\,\Rep_{\C}(G) \, \longrightarrow \, \Mod(X).
\end{equation*}
If $h\,:\,G \, \longrightarrow \, G'$ is a $\C$\nd homomorphism and $P'$ is the push forward of $P$ along 
$h_\mathrm{an}\,:\,G_\mathrm{an} \, \longrightarrow \, G'{}\!_\mathrm{an}$, then the composite 
\begin{equation}\label{e:hpush}
\Rep_{\C}(G') \,\xrightarrow{\,\,\,h^*\,\,} \, \Rep_{\C}(G) \, \xrightarrow{\,\,\,P \times^{G_\mathrm{an}} -\,\,} \, \Mod(X)
\end{equation} 
with the pullback $h^*$ along $h$ is tensor isomorphic to $P' \times^{G'{}\!_\mathrm{an}} -$. 

Let $P$ be a principal $G_\mathrm{an}$\nd bundle over $X$ and $T$ an exact tensor functor from $\Rep_{\C}(G)$ to $\Mod(X)$.
Then the action of $G_\mathrm{an}$ on $P$ where $g$ acts as $p \,\longmapsto\, pg^{-1}$ defines an action through $P$ on $T_P$, i.e.,\
a lifting of $T_P$ to $G_\mathrm{an}$\nd equivariant vector bundles over $P$.
We also have an action of $G_\mathrm{an}$ on $\omega_P$ given by the action on $P$ and the canonical action on $\omega$.
Let 
\begin{equation*}
\theta\,:\,\omega_P \,\iso\, T_P
\end{equation*}
be a tensor isomorphism which is compatible with the actions of $G_\mathrm{an}$, i.e.,\
for which each component $\theta_V$ is $G_\mathrm{an}$\nd equivariant.
Then $\theta$ defines a tensor isomorphism 
\begin{equation}\label{e:PGT}
P \times^{G_\mathrm{an}} - \,\iso\, T
\end{equation}
with component at $V$ the canonical isomorphism $P \times^{G_\mathrm{an}} V\,\iso\, T(V)$ defined by the $G_\mathrm{an}$\nd equivariant 
isomorphism $\theta_V$ from $\omega_P(V)\,=\, P \times V$ to the pullback $T_P(V)$ of $T(V)$ to $P$.

\begin{lem}\label{l:Pequ}
Let $X$ be a complex analytic space and $G$ an affine algebraic $\C$\nd group.
Then the functor that sends $P$ to $P \times^{G_\mathrm{an}} -$ from the category of principal $G_\mathrm{an}$\nd bundles over $X$ 
to the category of exact tensor functors $\Rep_{\C}(G) \,\longrightarrow\, \Mod(X)$ is an equivalence.
\end{lem}

\begin{proof}
By Lemma~\ref{l:principal} we have a functor $\underline{\Iso}_X^\otimes(\omega_X,\,-)$ from exact tensor functors
$\Rep_{\C}(G) \,\longrightarrow\, \Mod(X)$ to principal $G_\mathrm{an}$\nd bundles over $X$.

If $P$ is a principal $G_\mathrm{an}$\nd bundle over $X$, the canonical isomorphisms from $\omega_P(V)$ to the pullback of 
$P \times^{G_\mathrm{an}} V$ onto $P$ define a morphism
\begin{equation*}
P \,\longrightarrow\, \underline{\Iso}_X^\otimes(\omega_X,\,\,P \times^{G_\mathrm{an}} -)
\end{equation*}
of complex analytic spaces over $X$ which is natural in $P$.
It is compatible with the actions of $G_\mathrm{an}$, and hence is an isomorphism of principal $G_\mathrm{an}$\nd bundles.

Let $T\,:\,\Rep_{\C}(G) \,\longrightarrow\, \Mod(X)$ be an exact tensor functor.
If $P$ is the principal $G_\mathrm{an}$\nd bundle $\underline{\Iso}_X^\otimes(\omega_X,\,T)$ over $X$
and $\theta:\omega_P \iso T_P$ is the universal tensor isomorphism, 
then the right action of $G_\mathrm{an}$ on $P$ is by definition the unique morphism 
$b\,:\,P \times G_\mathrm{an} \,\longrightarrow\, P$ over $X$ such that the pullback of
$\theta$ along $b$ is the composite of the pullback of $\theta$ along the projection $P \times G_\mathrm{an}
\,\longrightarrow\, P$ with the pullback of the action automorphism of $\omega_{G_\mathrm{an}}$ along the projection 
$P \times G_\mathrm{an} \,\longrightarrow\, G_\mathrm{an}$. It follows that $\theta$ is compatible with the actions of $G_\mathrm{an}$
on $\omega_P$ and $T_P$. Hence $\theta$ defines as in \eqref{e:PGT} a tensor isomorphism
\begin{equation*}
\underline{\Iso}_X^\otimes(\omega_X,\,T) \times^{G_\mathrm{an}} - \,\,\iso\,\, T.
\end{equation*}
It is natural in $T$, because $\theta$ is functorial in $T$.

The functor that sends $P$ to $P \times^{G_\mathrm{an}} -$ has thus a quasi-inverse $\underline{\Iso}_X^\otimes(\omega_X,\,-)$.
\end{proof}

Let $G$ and $G'$ be affine algebraic $\C$\nd groups.
The $\C$\nd homomorphisms from $G'$ to $G$ form a category where a morphism from $h_1$ to $h_2$ is an element $g$ of $G(\C)$ for which
$h_2$ is the conjugate of $h_1$ by $g$.
We have a functor from this category to the category of faithful exact tensor functors
\begin{equation*}
\Rep_{\C}(G) \,\longrightarrow\, \Rep_{\C}(G')
\end{equation*}
which sends $h:G' \,\longrightarrow\, G$ to the pullback tensor functor $h^*$ along $h$ and $g:h_1 \,\longrightarrow\, h_2$ to the
tensor isomorphism with component at $V$ in $\Rep_{\C}(G)$ the action of $g$ on $V$.

\begin{lem}\label{l:Gequ}
Let $G$ and $G'$ be affine algebraic $\C$\nd groups.
Then the functor that sends $h$ to $h^*$ from the category $\C$\nd homomorphisms from $G'$ to $G$ to the
category of faithful exact tensor functors from $\Rep_{\C}(G)$ to $\Rep_{\C}(G')$ is an equivalence.
\end{lem}

\begin{proof}
Suppose first that $G' \,=\, 1$.
Then the full faithfulness follows from Lemma~\ref{l:Grep} and the essential surjectivity from the fact that by Lemma~\ref{l:principal}
with $X$ a point, every faithful exact tensor from $\Rep_{\C}(G)$ to $\Mod(\C)$ is tensor isomorphic to the forgetful functor $\omega$.

The full faithfulness for arbitrary $G'$ follows from that for $G' = 1$ by evaluating at a faithful representation of $G$.
To prove the essential surjectivity for arbitrary $G'$, 
it is enough by \cite[p.\,130, Corollary~2.9]{DM82} to show that every faithful exact tensor functor $H$ from $\Rep_{\C}(G)$ 
to $\Rep_{\C}(G')$ is tensor isomorphic to an $H'$ with
\begin{equation*}
\omega'H' \,= \,\omega,
\end{equation*}
where $\omega'$ is the forgetful tensor functor from $\Rep_{\C}(G')$.
By the case where $G' = 1$, there is a tensor isomorphism $\varphi_0$ from $\omega'H$ to $\omega$.
For each $V$ in $\Rep_{\C}(G)$, there is by transport of structure a unique isomorphism $\varphi_V$ in $\Rep_{\C}(G')$
with source $H(V)$ such that $\omega'(\varphi_V) \,=\, \varphi_0{}_V$.
There is then a unique tensor functor $H'$ from $\Rep_{\C}(G)$ to $\Rep_{\C}(G')$ such that the $\varphi_V$ are the components
of a tensor isomorphism $\varphi:H \,\iso\, H'$.
Since $\omega' \varphi \,=\, \varphi_0$, we have $\omega'H'\, =\, \omega$.
\end{proof}

Theorem~\ref{t:min} below is an application, using Lemmas~\ref{l:Pequ} and \ref{l:Gequ}, of the splitting theorem for tensor categories
proved in \cite{AndKah}, \cite{O} and \cite{O19}.

Let $\sC$ be a tensor category.
If $\sC$ rigid, in the sense that every object of $\sC$ has a dual,
then the trace $\tr(e) \,\in\, \End_{\sC}(\I)$ of an endomorphism $e$ in $\sC$ is defined.
Suppose that $\sC$ is rigid, and that $\End_{\sC}(\I)$ is a local $\C$\nd algebra.
Then the morphisms $r\,:\,A \,\longrightarrow\, B$ in
$\sC$ for which $\tr(s \circ r)$ lies in the maximal ideal of $\End_{\sC}(\I)$ for every $s\,:\,B\,\longrightarrow\, A$ 
form the unique maximal tensor ideal $\sJ$ of $\sC$ \cite[p. 73]{O19}.
We write
\begin{equation*}
\overline{\sC}
\end{equation*} 
for the quotient of $\sC$ by $\sJ$.
It is a rigid tensor category with the same objects as $\sC$, and there is a projection $\sC\,\longrightarrow\, \overline{\sC}$
which is the identity map on objects and is a full tensor functor with kernel $\sJ$.
The restriction of $\sJ$ to any full rigid tensor subcategory $\sC_0$ of $\sC$ 
is the unique maximal tensor ideal of $\sC_0$, with $\overline{\sC_0}$ a full tensor subcategory of $\overline{\sC}$.

The group algebra of the symmetric group of degree $d$ acts on the $d$th tensor power of every object of $\sC$,
and if $\sC$ is pseudo-abelian
we may define for example the $d$th exterior power as the image of the antisymmetrising idempotent.

Suppose that $\sC$ is essentially small (i.e.,\ has a small skeleton), pseudo-abelian and rigid, that the $\C$\nd algebra
$\End_{\sC}(\I)$ is henselian local with residue field $\C$, and that for every object of $\sC$ some exterior power is $0$.
Then $\overline{\sC}$ is a semisimple Tannakian category over $\C$
(see \cite[Proposition~11.2(i)]{O19} and \cite[p. 165, Th\'eor\`eme~7.1]{Del90}). 
The splitting theorem for $\sC$ (see \cite[Theorem~11.7]{O19}) states that the projection $Q\,:\,\sC
\,\longrightarrow\, \overline{\sC}$ has a right quasi-inverse,
i.e.,\ a tensor functor $T\,:\,\overline{\sC}\,\longrightarrow\, \sC$ with the composite
\begin{equation*}
\overline{\sC} \,\xrightarrow{\,\,T\,\,} \, \sC \,\xrightarrow{\,\,Q\,\,}\, \overline{\sC}
\end{equation*}
tensor isomorphic to the identity.
Further if $T$ is such a right quasi-inverse and $\sD$ is an essentially small rigid tensor category,
then any tensor functor $S\,:\,\sD \,\longrightarrow\, \sC$ with $QS\,:\,\sD \,\longrightarrow\, \overline{\sC}$
faithful factors as
\begin{equation*}
\sD \,\longrightarrow\, \overline{\sC} \,\xrightarrow{\,\,T\,\,} \, \sC
\end{equation*} 
up to tensor isomorphism.

We recall from \cite[Lemma~2.1]{BisO'S21} that for any non-empty complex analytic space $X$, the $\C$\nd algebra $H^0(X,\,\sO_X)$ is 
henselian local with residue field $\C$ if and only if the restriction to $X_\mathrm{red}$ of every holomorphic function on $X$ is constant.
When this is so, the hypotheses on $\sC$ in the preceding paragraph are satisfied for
every full pseudo-abelian rigid tensor subcategory $\sC$ of $\Mod(X)$.

Recall \cite[Remark~3.10]{DM82} that any finitely generated Tannakian category over $\C$ is neutral,
and hence tensor equivalent to $\Rep_{\C}(G)$ for some affine $\C$\nd group $G$, necessarily algebraic.
Further $\Rep_{\C}(G)$ is semisimple if and only if $G$ is reductive, and when this is so, any tensor functor from $\Rep_{\C}(G)$
to a non-zero tensor category is faithful and exact.

Let $X$ be a complex analytic space and $J$ a reductive complex Lie group.
A principal $J$\nd bundle $P$ over $X$ will be called \emph{minimal} if a principal $J'$\nd subbundle of $P$ exists for a 
reductive closed complex Lie subgroup $J'$ of $J$ only when $J' \,=\, J$.
Any push forward of a minimal principal $J$\nd bundle over $X$ along a surjective homomorphism of algebraic complex Lie groups is minimal. 

\begin{thm}\label{t:min}
Let $X$ be a non-empty complex analytic space, $G$ a reductive algebraic $\C$\nd group and $P$ a principal $G_\mathrm{an}$\nd bundle over $X$.
Suppose that the restriction to $X_\mathrm{red}$ of every holomorphic function on $X$ is constant. 
Then $P$ is minimal if and only if the composite of the projection $\Mod(X) \,\xrightarrow{\,\,\,} \, \overline{\Mod(X)}$ with 
$P \times^{G_\mathrm{an}} -:\Rep_{\C}(G) \,\xrightarrow{\,\,\,} \, \Mod(X)$ is fully faithful.
\end{thm}

\begin{proof}
Suppose that $P$ is minimal.
Since $G$ is a reductive algebraic $\C$\nd group, $\Rep_{\C}(G)$ is finitely generated as a pseudo-abelian rigid tensor category.
Therefore, $P \times^{G_\mathrm{an}} -$ factors through a finitely generated pseudo-abelian full rigid tensor subcategory 
$\sC$ of $\Mod(X)$. 
Further any tensor functor from $\Rep_{\C}(G)$ to a non-zero tensor category is faithful and exact.
By the splitting theorem, $P \times^{G_\mathrm{an}} -$ thus factors up to tensor isomorphism as
\begin{equation*}
\Rep_{\C}(G) \,\longrightarrow \, \overline{\sC}\, \xrightarrow{\,\,\,T\,\,}\, \sC \, \longrightarrow \, \Mod(X)
\end{equation*} 
where the third arrow is the embedding and $T$ followed by the projection onto $\overline{\sC}$ is an equivalence.
Now $\overline{\sC}$ is a finitely generated semisimple Tannakian category over $\C$, and hence as above $\C$\nd tensor equivalent to 
$\Rep_{\C}(G_0)$ for some reductive algebraic $\C$\nd group $G_0$. 
Using the essential surjectivity of Lemmas~\ref{l:Pequ} and \ref{l:Gequ}, we thus obtain for some $\C$\nd homomorphism 
$h\,:\,G_0 \, \longrightarrow \, G$ and principal $G_0$\nd bundle $P_0$ a factorisation
\begin{equation}\label{e:Pfac}
\Rep_{\C}(G) \,\xrightarrow{\,\,\,h^*\,\,} \, \Rep_{\C}(G_0) \, \xrightarrow{\,\,\,P_0 \times^{G_0{}_\mathrm{an}} -\,\,} \, \Mod(X)
\end{equation} 
of $P \times^{G_\mathrm{an}} -$ up to tensor isomorphism, with $P_0 \times^{G_0{}_\mathrm{an}} -$ followed by the projection onto
$\overline{\Mod(X)}$ fully faithful. 
As in \eqref{e:hpush}, the composite \eqref{e:Pfac} is also tensor isomorphic to $P_1 \times^{G_\mathrm{an}} -$, 
where $P_1$ is the push forward of $P_0$ along $h_\mathrm{an}$.
By the full faithfulness of Lemma~\ref{l:Pequ}, $P_1$ is thus isomorphic to $P$.
Hence since $P$ is minimal, $h$ is surjective, so that $h^*$ is fully faithful.
The composite of the projection onto $\overline{\Mod(X)}$ with \eqref{e:Pfac} is then fully faithful, as required. 

Conversely suppose that the composite of the projection with $P \times^{G_\mathrm{an}} -$ is fully faithful.
It is to be shown that $G' \,=\, G$ for every reductive $\C$\nd subgroup $G'$ of $G$ for which $P$ has a principal 
$G'{}\!_\mathrm{an}$\nd subbundle $P'$.
If $i\,:\,G' \, \longrightarrow \, G$ is the embedding, then as in \eqref{e:hpush}, $P \times^{G_\mathrm{an}} -$ factors up to tensor isomorphism as
\begin{equation*}
\Rep_{\C}(G) \,\xrightarrow{\,\,\,i^*\,\,} \, \Rep_{\C}(G') \, \xrightarrow{\,\,\,P' \times^{G'{}\!_\mathrm{an}} -\,\,} \, \Mod(X).
\end{equation*} 
Composing with the projection and using the fact that $\Rep_{\C}(G') \, \longrightarrow \, \overline{\Mod(X)}$ is faithful
shows that $i^*$ is fully faithful.
Hence $V^G\, =\, V^{G'}$ for every representation $V$ of $G$.
Taking for $V$ the coordinate algebra of $G$ with $G$ acting by right translation shows that the affine $\C$\nd scheme $G/G'$ is a point.
Thus $G' \,=\, G$.
\end{proof}

\begin{thm}\label{t:pushan}
Let $X$ be a complex analytic space, and let $J_0$ and $J$ be algebraic complex Lie groups 
with $J_0$ reductive. Let $P_0$ be a minimal principal $J_0$\nd bundle over $X$, and let 
$h_1$ and $h_2$ be complex analytic homomorphisms from $J_0$ to $J$. Suppose that the 
restriction to $X_\mathrm{red}$ of every holomorphic function on $X$ is constant. Then the 
push forwards of $P_0$ along $h_1$ and $h_2$ are isomorphic if and only if $h_1$ and $h_2$ 
are conjugate.
\end{thm}

\begin{proof}
The ``if'' part is clear, even without any condition on $J_0$ or $X$.

To prove the converse, we may assume that $J_0 = G_0{}_\mathrm{an}$ and $J = G_\mathrm{an}$ for affine algebraic $\C$\nd groups $G_0$ and $G$, 
with $G_0$ reductive.
Then for $i\,=\, 1,\,2$ we have $h_i = l_i{}_\mathrm{an}$ for a $\C$\nd homomorphism $l_i\,:\,G_0 \, \longrightarrow \, G$.
Suppose that the push forwards of $P_0$ along $l_1{}_\mathrm{an}$ and $l_2{}_\mathrm{an}$ are isomorphic.
Then the composites of 
\begin{equation*}
P_0 \times^{G_0{}_\mathrm{an}} -\,:\,\Rep_{\C}(G_0) \, \longrightarrow \, \Mod(X)
\end{equation*} 
with $l_1{}\!^*$ and $l_2{}\!^*$ are tensor isomorphic.
Composing with the projection onto $\overline{\Mod(X)}$ and using Theorem~\ref{t:min} shows that $l_1{}\!^*$ and $l_2{}\!^*$ are tensor isomorphic.
Thus by the full faithfulness of Lemma~\ref{l:Gequ}, $l_1$ and $l_2$ and hence $h_1$ and $h_2$ are conjugate.
\end{proof}

Corollary~\ref{c:conjan} below follows from Theorem~\ref{t:pushan} in the same way as Corollary~\ref{c:conj} from Theorem~\ref{t:push}.

\begin{cor}\label{c:conjan}
Let $X$ be a complex analytic space, $J$ an algebraic complex Lie group and $P$ a principal $J$\nd bundle over $X$.
For $i\,=\, 0,\,1$, let $J_i$ be a reductive closed complex Lie subgroup of $P$ and $P_i$ a principal $J_i$\nd subbundle of $P$.
Suppose that $P_0$ is minimal, and that the restriction to $X_\mathrm{red}$ of every holomorphic function on $X$ is constant.
Then there exists an element $j \,\in\, J$ such that $J_1$ contains $jJ_0j^{-1}$ and $P_1$ contains a principal $jJ_0j^{-1}$\nd subbundle
isomorphic to $P_0j^{-1}$. 
\end{cor}

A representation of a complex Lie group $J$ on a finite-dimensional $\C$\nd vector space $V$ may be identified with 
a homomorphism of complex Lie groups from $J$ to $\underline{\Aut}_{\C}(V)_{\mathrm{an}}$.
Therefore, if $G$ is a reductive algebraic $\C$\nd group, restriction along $G_\mathrm{an}
\,\longrightarrow\, G$ defines an equivalence (even an isomorphism)
from $\Rep_{\C}(G)$ to finite-dimensional representations of $G_\mathrm{an}$.

As in the algebraic case, we have for any complex analytic space $X$ an equivalence from vector bundles of rank $n$ over $X$ to principal
$({\rm GL}_n)_\mathrm{an}$\nd bundles over $X$, with decomposition as a direct sum of vector bundles of ranks $r$ and $s$ for $r + s
\,=\, n$ corresponding to 
reduction of the structure group from $({\rm GL}_n)_\mathrm{an}$ to $({\rm GL}_r)_\mathrm{an}\times ({\rm GL}_s)_\mathrm{an}$.
Similarly $P \times^J V_1$ and $P \times^J V_2$ are isomorphic for $n$\nd dimensional representations $V_1$ and $V_2$
of a complex Lie group $J$ if and only if the push forwards of the principal $J$\nd bundle $P$ over $X$ along 
$J\,\longrightarrow\, ({\rm GL}_n)_\mathrm{an}$ corresponding to $V_1$ and $V_2$ after choosing bases are isomorphic. 

Corollary~\ref{c:assocan} below follows from Theorem~\ref{t:pushan} and Corollary~\ref{c:conjan} in the same way as 
Corollary~\ref{c:assoc} from Theorem~\ref{t:push} and Corollary~\ref{c:conj}.
Alternatively it can be deduced directly from Theorem~\ref{t:min}. 

\begin{cor}\label{c:assocan}
Let $X$ be a complex analytic space, $J$ a reductive complex Lie group and $P$ a minimal principal $J$\nd bundle over $X$.
Suppose that the restriction to $X_\mathrm{red}$ of every holomorphic function on $X$ is constant. Then the following
two hold:
\begin{enumerate}
\item
For any finite-dimensional representation $V$ of $J$, the vector bundle $P \times^J V$ over $X$ is indecomposable if and only if 
$V$ is irreducible.

\item
For any two finite-dimensional representations $V_1$ and $V_2$ of $J$, the vector bundles $P \times^J V_1$ and 
$P \times^J V_2$ over $X$ are isomorphic if and only if $V_1$ and $V_2$ are.
\end{enumerate}
\end{cor}

Let $G$ be an affine algebraic $\C$\nd group, and $G'\, \subset\, G$ be a $\C$\nd subgroup.
Right translation by $G'$ defines a structure of principal $G'$\nd bundle over $G/G'$ on $G$, and 
left translation by $G$ then defines a $G$\nd equivariant structure on the principal $G'$\nd bundle $G$ 
over $G/G'$. If $V'$ is a representation of $G'$, the
$G$\nd equivariant structure on the principal $G'$\nd bundle $G$ produces
a $G$\nd equivariant structure on the quasi-coherent $\sO_{G/G'}$\nd module $G \times_{\C}^{G'} V'$.
Recall that the induction functor from $\mathrm{REP}_{\C}(G')$ to $\mathrm{REP}_{\C}(G)$ is defined by
\begin{equation*}
\mathrm{Ind}_{G'}^G(V') \,\,=\,\, H^0(G/G',\,G \times_{\C}^{G'} V').
\end{equation*} 
It is right adjoint to the restriction functor from $\mathrm{REP}_{\C}(G)$ to
$\mathrm{REP}_{\C}(G')$, with adjunction isomorphism
\begin{equation*}
\Hom_{G'}(V,\,V')\,\iso\, \Hom_{G,G/G'}((G/G')\times_{\C} V,\,G \times_{\C}^{G'} V')
\,\iso\,\Hom_G(V,\,\mathrm{Ind}_{G'}^G(V')),
\end{equation*}
where if we denote by $\overline{g}$ the image of any $g$ of $G$ in $G/G'$, and by
$\overline{(g,\, v)}$ the image of a
point $(g,\,v)$ of $G \times_{\C} V'$ in $G \times_{\C}^{G'} V'$,
the first isomorphism sends any $\lambda\,:\,V \,\longrightarrow\, V'$ to the
homomorphism defined by
\begin{equation*}
(\overline{g},\,v) \,\,\longmapsto\,\, \overline{(g,\,\lambda(g^{-1}v))},
\end{equation*}
with the inverse given by taking the fibre at the base point of $G/G'$.

Suppose that $G/G'$ is affine.
Then $\mathrm{Ind}_{G'}^G$ is exact. Consequently, we have an isomorphism
\begin{equation}\label{e:Ext}
\Ext_{G'}^1(V,\,V') \,\,\iso\,\, \Ext_G^1(V,\,\mathrm{Ind}_{G'}^G(V'))
\end{equation}
for $V$ in $\REP_{\C}(G)$ and $V'$ in $\REP_{\C}(G')$, given by applying
$\mathrm{Ind}_{G'}^G$ to
an extension and then pulling back along the unit $V\,\longrightarrow\, \mathrm{Ind}_{G'}^G(V)$,
with the inverse given by applying the restriction functor and then pushing forward along the counit $\mathrm{Ind}_{G'}^G(V')
\,\longrightarrow\, V'$. Taking $V \,=\, \C$ in \eqref{e:Ext} shows that every extension $W'$ of $\C$ by $V'$ in $\REP_{\C}(G')$ 
is the push forward of the restriction from $G$ to $G'$ of an extension $W$ of $\C$ in $\REP_{\C}(G)$.
Further if $V'$ is finite-dimensional, replacing $W$ by a sufficiently large subrepresentation
shows that $W$ may be taken to be finite dimensional. 

\begin{prop}\label{p:splan}
Let $X$ be a complex analytic space, $G$ an affine algebraic $\C$\nd group and $P$ a principal $G_\mathrm{an}$\nd bundle over $X$.
Then $P$ has a principal $G_0{}_\mathrm{an}$\nd subbundle for some reductive $\C$\nd subgroup $G_0$ of $G$ if and only if 
the functor $P \times^{G_\mathrm{an}} -$ on $\Rep_{\C}(G)$ splits every short exact sequence.
\end{prop}

\begin{proof}
The ``only if'' part follows from the fact that $P \times^{G_\mathrm{an}} -$ factors through $P \times^{G_0{}_\mathrm{an}} -$
for any $\C$\nd subgroup $G_0$ of $G$.

Conversely, suppose that $P \times^{G_\mathrm{an}} -$ splits every short exact sequence of finite-dimensional representation of $G$.
Write
\begin{equation*}
G \,=\, U \rtimes_{\C} G_0,
\end{equation*}
where $U$ is the unipotent radical of $G$ and $G_0$ is a reductive $\C$\nd subgroup of $G$.
Then we show that $P$ has a principal $G_0{}_\mathrm{an}$\nd subbundle. 

We argue by induction on the length of the lower central series
\begin{equation*}
U \,=\, U_1 \,\supset\, U_2 \,\supset\, \cdots \,\supset\, U_n \,\supset\, U_{n+1}\, =\, 1
\end{equation*}
of $U$. If $n \,=\, 0$ then $G_0 \,=\, G$. Suppose that $n \,>\, 0$.
Then the splitting condition for $P \times^{G_\mathrm{an}} -$ is satisfied with $G$ replaced by $G/U_n$ and $P$ by the push forward 
$P/(U_n)_\mathrm{an}$ of $P$ along the projection from $G_\mathrm{an}$ onto $(G/U_n)_\mathrm{an}$, because the relevant functor is isomorphic
to inflation from $G/U_n $ to $G$ followed by $P \times^{G_\mathrm{an}} -$. Since
\begin{equation*}
G/U_n \,=\, (U/U_n) \rtimes_{\C} G_0,
\end{equation*}
$P/(U_n)_\mathrm{an}$ has by the induction hypothesis a principal $G_0{}_\mathrm{an}$\nd subbundle.
Its inverse image under the projection from $P$ to $P/(U_n)_\mathrm{an}$ is a principal $G'{}\!_\mathrm{an}$\nd subbundle $P'$ of $P$, where 
\begin{equation*}
G'\, =\, U_n \rtimes_{\C} G_0
\end{equation*}
is the inverse image of $G_0$ under the projection from $G$ to $G/U_n$.
We show that $P'/G_0{}_\mathrm{an}$ has a cross-section over $X$, so that $P'$ and hence $P$ has a principal $G_0{}_\mathrm{an}$\nd subbundle.

Since $U_n$ is commutative and unipotent, it may be identified with a finite-dimensional $\C$\nd vector space.
The action of $G'$ on $U_n$ by conjugation then gives it a structure of representation of $G'$.
We have a short exact sequence
\begin{equation*}
0 \,\longrightarrow\, U_n \,\longrightarrow\, W' \,\longrightarrow\, \C \,\longrightarrow\, 0
\end{equation*}
of representations of $G'$, where $W'$ has underlying $\C$\nd vector space $U_n \oplus k$ with action
\begin{equation*}
(ug_0)(v,\,\alpha) \,=\, (g_0v + \alpha u,\,\alpha)
\end{equation*}
of the point $ug_0$ of $G'$ on the point $(v,\,\alpha)$ of $W'$, and $G'$ acts trivially on $\C$.
The fibre of $W'$ above $1$ in $\C$ may then be identified with
\begin{equation*}
G'/G_0 \,=\, U_n
\end{equation*} 
on which $G'$ acts by translation. 
Thus we have short exact sequence
\begin{equation}\label{esp}
0 \,\longrightarrow\, P' \times^{G'{}\!_\mathrm{an}} U_n \,\longrightarrow\, P' \times^{G'{}\!_\mathrm{an}} W'
\,\longrightarrow\, P' \times^{G'{}\!_\mathrm{an}} \C \,\longrightarrow\, 0
\end{equation}
of vector bundles over $X$, where the fibre of $P' \times^{G'{}\!_\mathrm{an}} W'$ above the identity cross-section of 
$P' \times^{G'{}\!_\mathrm{an}} \C \,=\, X \times \C$ is
\begin{equation*}
P' \times^{G'{}\!_\mathrm{an}} (G'/G_0)_\mathrm{an} \,=\, P'/G_0{}_\mathrm{an}.
\end{equation*}
It is thus enough to show that the sequence in \eqref{esp} splits.
In fact $G/G' \,=\, U/U_n$ is affine, so that as above the extension
$W'$ of $\C$ by $U_n$ is the push forward of the restriction from $G$ to $G'$ of an extension in $\Rep_{\C}(G)$.
Since restriction from $G$ to $G'$ followed by $P' \times^{G'{}\!_\mathrm{an}} -$ is isomorphic to $P \times^{G_\mathrm{an}} -$,
the required splitting follows from that assumed for $P \times^{G_\mathrm{an}} -$. 
\end{proof}

Let $f\,:\,X'\,\longrightarrow\, X$ be a proper surjective morphism of complex analytic spaces, with $X$ irreducible.
Then the restriction of $f$ to some irreducible component $X''$ of $X'$ is surjective:
the covering of $X'$ by its irreducible components is locally finite \cite[9.2.2]{GraRem84},
so that if $X_\mathrm{red}$ is smooth at $x \,\in\, X$, there is an irreducible component $X''$ of $X'$ for which the dimensions of $f(X'')$
and $X$ coincide at $x$, and hence $f(X'') \,=\, X$ \cite[9.1.1]{GraRem84}.

Theorem~\ref{t:redsuban} below follows from Lemmas~\ref{l:restran}, \ref{l:isoan} and \ref{l:inj} and Proposition~\ref{p:splan} 
in the same way that Theorem~\ref{t:redsub} follows from Lemmas~\ref{l:restr}, \ref{l:iso} and \ref{l:inj} and Proposition~\ref{p:spl}:
After writing $J \,=\, G_\mathrm{an}$, the proof in the case where \ref{i:redsublfan} holds is almost identical, 
and in the case where \ref{i:redsubpropan} holds, 
we again write $X$ 
as the disjoint union $\coprod_\alpha X_\alpha$ of its irreducible components $X_\alpha$, and replace $f$ by its restriction to 
$\coprod_\alpha X'{}\!_\alpha$, where $X'{}\!_\alpha$ is as above an irreducible component of $f^{-1}(X_\alpha)$ with
$f(X'{}\!_\alpha) \,=\, X_\alpha$.

\begin{thm}\label{t:redsuban}
Let $f\,:\,X'\,\longrightarrow\, X$ be a morphism of complex analytic spaces, $J$ an algebraic complex Lie group and $P$ a principal 
$J$\nd bundle over $X$. Suppose that either of the following conditions holds:
\begin{enumerate}
\renewcommand{\theenumi}{(\alph{enumi})}
\item\label{i:redsublfan}
$f_*\sO_{X'}$ is a locally free $\sO_X$\nd module of finite type which is nowhere $0$;
\item\label{i:redsubpropan} 
$X$ is normal, and $f$ is proper and surjective.
\end{enumerate}
Then $P$ has a principal $J_0$\nd subbundle for some reductive closed complex Lie subgroup $J_0$ of $J$ if and only if $f^*P$ does. 
\end{thm}

Let $X$ be a complex analytic space and $J$ a reductive complex Lie group.
A principal $J$\nd bundle $P$ over $X$ will be called \emph{almost minimal} if a principal $J'$\nd subbundle of $P$ exists for a 
reductive closed complex Lie subgroup $J'$ of $J$ only if $J'$ contains the identity component of $J$.
Any push forward of an almost minimal principal $J$\nd bundle over $X$ along a surjective homomorphism of algebraic complex Lie groups 
is almost minimal. 

Lemma~\ref{l:alminan} below can be proved in the same way as Lemma~\ref{l:almin}.

\begin{lem}\label{l:alminan}
Let $X$ be a complex analytic space, $J$ a reductive complex Lie group and $J_0\, \subset\,J$ a closed complex Lie subgroup containing 
the identity component of $J$. Let $P$ be a principal $J$\nd bundle over $X$ and $P_0$ a principal $J_0$\nd subbundle of $P$.
Then $P$ is almost minimal if and only if $P_0$ is so.
\end{lem}

Theorem~\ref{t:alminan} below follows from Lemmas~\ref{l:restran}, \ref{l:isoan}, \ref{l:rep} and \ref{l:alminan}
and Corollary~\ref{c:assocan} in the same way that Theorem~\ref{t:almin} follows from Lemmas~\ref{l:restr}, 
\ref{l:iso}, \ref{l:rep} and \ref{l:almin} and Corollary~\ref{c:assoc}.
Note that if \ref{i:redsubpropan} of Theorem~\ref{t:redsuban} holds, then $X$ is reduced, so that for $X$ non-empty the condition
``the restriction to $X_\mathrm{red}$ of every holomorphic function on $X$ is constant'' is equivalent to ``$H^0(X,\,\sO_X) \,=\, \C$''.

\begin{thm}\label{t:alminan}
Let $f\,:\,X'\,\longrightarrow\, X$ be a morphism of complex analytic spaces, $J$ a reductive complex Lie group
and $P$ a principal $J$\nd bundle over $X$.
Suppose that the restriction to $X_\mathrm{red}$ of every holomorphic function on $X$ is constant, and that either 
\ref{i:redsublfan} or \ref{i:redsubpropan} of Theorem~\ref{t:redsuban} holds.
Then $P$ is almost minimal if and only if $f^*P$ is so.
\end{thm}

Let $f\,:\,X'\,\longrightarrow\, X$ be a morphism of complex analytic spaces for which either 
\ref{i:redsublfan} or \ref{i:redsubpropan} of Theorem~\ref{t:redsuban} holds.
Then $f(X')$ is dense in $X$.
Suppose that the restriction to $X'{}\!_\mathrm{red}$ of every holomorphic function on $X'$ is constant.
Then the restriction to $X_\mathrm{red}$ of every holomorphic function on $X$ is also constant.
The complex analytic analogues of the parametrisation \eqref{e:Tset} and the map \eqref{e:Tmap} then follow from
Theorem~\ref{t:pushan} and Corollary~\ref{c:conjan} in the same way as \eqref{e:Tset} and \eqref{e:Tmap} from
Theorem~\ref{t:push} and Corollary~\ref{c:conj}.
Thus if $J$ is an algebraic complex Lie group and $P$ is a principal $J$\nd bundle over $X$,
and if $P$ has for some reductive closed complex Lie subgroup $J_0$ of $J$ a minimal principal $J_0$\nd subbundle $P_0$ with $f^*P_0$ also
minimal, then pullback along $f$ induces for any reductive closed complex Lie subgroup $J_1$ of $J$ a bijection from
the isomorphism classes of principal $J_1$\nd subbundles of $P$ to the isomorphism classes of principal $J_1$\nd subbundles of $f^*P$.
Theorem~\ref{t:pullan} below then follows from Theorems~\ref{t:redsuban} and \ref{t:alminan} in the same way as Theorem~\ref{t:pull} 
from Theorems~\ref{t:redsub} and \ref{t:almin}.

\begin{thm}\label{t:pullan}
Let $f\,:\,X'\,\longrightarrow\, X$ be a morphism of complex analytic spaces.
Suppose that the restriction to $X'{}\!_\mathrm{red}$ of every holomorphic function on $X'$ is constant,
that $f^*Z$ is connected for every connected finite \'etale cover $Z$ of $X$, and that either 
\ref{i:redsublfan} or \ref{i:redsubpropan} of Theorem~\ref{t:redsuban} holds.
Then for every algebraic complex Lie group $J$, reductive closed complex Lie subgroup $J_1$ of $J$,
and principal $J$\nd bundle $P$ over $X$, pullback along $f$ defines a bijection 
from the set isomorphism classes of principal $J_1$\nd subbundles of $P$
to the set isomorphism classes of principal $J_1$\nd subbundles of $f^*P$.
\end{thm}

Corollary~\ref{c:pullan} below follows from Theorem~\ref{t:pullan} in the same way as Corollary~\ref{c:pull} from Theorem~\ref{t:pull}.
Alternatively if $f\,:\,X'\,\longrightarrow\, X$ is as in Theorem~\ref{t:pullan}, then by Theorems~\ref{t:min} and \ref{t:pullan} 
the composite of the projection from $\Mod(X')$ to $\overline{\Mod(X')}$ with the pullback tensor functor
\begin{equation*}
f^*\,:\,\Mod(X) \, \,\longrightarrow \,\, \Mod(X')
\end{equation*}
is full, so that $f^*$ is ``local'', i.e.,\ sends the unique maximal tensor ideal of $\Mod(X)$ into the unique maximal tensor ideal 
of $\Mod(X')$, and the induced tensor functor
\begin{equation*}
\overline{\Mod(X)} \,\, \longrightarrow \,\, \overline{\Mod(X')}
\end{equation*}
is fully faithful.
Corollary~\ref{c:pullan} follows, in the case of \ref{i:pullPan} using the splitting theorem.

\begin{cor}\label{c:pullan}
Let $f\,:\,X'\,\longrightarrow\, X$ be a morphism of complex analytic spaces such that the hypotheses of Theorem~\ref{t:pullan} are satisfied.
\begin{enumerate}
\item\label{i:pullPan}
For any reductive complex Lie group $J$, two principal $J$\nd bundles $P_1$ and $P_2$ over $X$ are isomorphic
if and only if $f^*P_1$ and $f^*P_2$ are so.
\item\label{i:pullVindecan}
A vector bundle $\sV$ over $X$ is indecomposable if and only if $f^*\sV$ is so.
\item\label{i:pullVisoan}
Two vector bundles $\sV_1$ and $\sV_2$ over $X$ are isomorphic if and only if $f^*\sV_1$ and $f^*\sV_2$ are so.
\end{enumerate}
\end{cor}

Consider the following condition which is slightly stronger than \ref{i:redsubpropan} of Theorem~\ref{t:redsuban}:

\begin{itemize}
\item[$\mathrm{\ref{i:redsubprop}}^\prime$]
$X$ is normal, $X'$ is non-empty, and $f$ is proper with restriction to every 
irreducible component of $X'$ surjective.
\end{itemize}

It can be seen in a similar way to the algebraic case that if in Theorem~\ref{t:pullan} the hypothesis that 
``either \ref{i:redsublfan} or \ref{i:redsubpropan} of Theorem~\ref{t:redsuban} holds'' is replaced by
``either \ref{i:redsublfan} of Theorem~\ref{t:redsuban} or $\mathrm{\ref{i:redsubprop}}^\prime$ holds'',
then the hypothesis ``the restriction to $X'{}\!_\mathrm{red}$ of every holomorphic function on $X'$ is constant'',
may be replaced by ``the restriction to $X_\mathrm{red}$ of every holomorphic function on $X$ is constant''.

\section*{Acknowledgements}

The first author is partially supported by a J. C. Bose Fellowship (JBR/2023/000003).

\providecommand{\bysame}{\leavevmode\hbox to3em{\hrulefill}\thinspace}
\providecommand{\MR}{\relax\ifhmode\unskip\space\fi MR }
\providecommand{\MRhref}[2]{%
\href{http://www.ams.org/mathscinet-getitem?mr=#1}{#2}
}
\providecommand{\href}[2]{#2}

\end{document}